\documentclass[11pt]{amsart}

\usepackage{amssymb}
\usepackage{bm}
\usepackage[centertags]{amsmath}
\usepackage{amsfonts}
\usepackage{amsthm}
\usepackage{mathrsfs}
\usepackage[usenames]{xcolor}

\theoremstyle{definition}
\newtheorem{thm}{Theorem}[section]
\newtheorem{dfn}[thm]{Definition}
\newtheorem{prp}[thm]{Proposition}
\newtheorem{lem}[thm]{Lemma}
\newtheorem{cor}[thm]{Corollary}
\newtheorem{rmk}[thm]{Remark}
\newtheorem{ntt}[thm]{Notation}

\newtheorem*{thm*}{Theorem}

\newtheorem*{dfn*}{Definition}

\newtheorem*{cor*}{Corollary}

\newtheorem*{prp*}{Proposition}

\newtheorem*{prb*}{Problem}

\newcommand{\Ga}{\Gamma}
\newcommand{\ga}{\gamma}
\newcommand{\De}{\Delta}
\newcommand{\inn}{\in\mathbb{N}}
\newcommand{\al}{\alpha}
\newcommand{\de}{\delta}
\newcommand{\la}{\lambda}

\newcommand{\e}{\varepsilon}
\newcommand{\N}{\mathbb{N}}
\newcommand{\R}{\mathbb{R}}

\newcommand{\adxk}{\al((x_k)_k)}

\newcommand{\adyk}{\al((y_k)_k)}

\newcommand{\adyq}{\al((y_q)_q)}

\newcommand{\adzk}{\al((z_k)_k)}

\newcommand{\aduk}{\al((w_k)_k)}

\newcommand{\ac}{$\al_c$}
\newcommand{\ic}{$\mathcal{IC}$}
\newcommand{\co}{$\mathcal{CO}$}
\newcommand{\ir}{$\mathcal{IR}$}

\newcommand{\Xbd}{\mathfrak{X}_{(\Ga_q, i_q)_q}}
\newcommand{\Xq}{\mathfrak{X}_{(\Ga_{q_s}', i_{q_s}')_s}}
\newcommand{\BmT}{\mathfrak{B}_{\mathrm{mT}}}
\newcommand{\X}{\mathfrak{X}_{\mathfrak{nr}}}
\newcommand{\Xstar}{\mathfrak{X}^*_{\mathfrak{nr}}}

\DeclareMathOperator{\supp}{supp}

\DeclareMathOperator{\ran}{ran}

\DeclareMathOperator{\ra}{rank}

\DeclareMathOperator{\ag}{age}

\DeclareMathOperator{\we}{weight}

\long\def\symbolfootnote[#1]#2{\begingroup
\def\thefootnote{\fnsymbol{footnote}}\footnote[#1]{#2}\endgroup}

\allowdisplaybreaks

\begin{document}

\title[Scalar-plus-compact property without reflexive subspaces]{The scalar-plus-compact property in spaces without reflexive subspaces}

\author[S. A. Argyros]{Spiros A. Argyros}
\address{National Technical University of Athens, Faculty of Applied Sciences,
Department of Mathematics, Zografou Campus, 157 80, Athens, Greece}
\email{sargyros@math.ntua.gr}

\author[P. Motakis]{Pavlos Motakis}
\address{Department of Mathematics, Texas A\&M University, College Station, TX 77843-3368, U.S.A.}
\email{pavlos@math.tamu.edu}


\thanks{{\em 2010 Mathematics Subject Classification:} Primary 46B03, 46B06, 46B25, 46B45.}
\thanks{The second author's research was supported by NSF DMS-1600600.}

\begin{abstract}
A hereditarily indecomposable Banach space $\X$ is constructed that is the first known example of a $\mathscr{L}_\infty$-space not containing $c_0$, $\ell_1$, or reflexive subspaces and answers a question posed by J. Bourgain. Moreover, the space $\X$ satisfies the ``scalar-plus-compact'' property and it is the first known space without reflexive subspaces having this property. It is constructed using the Bourgain-Delbaen method in combination with a recent version of saturation under constraints in a mixed-Tsirelson setting. As a result, the space $\X$ has a shrinking finite dimensional decomposition and does not contain a boundedly complete sequence.
\end{abstract}

\maketitle


\section*{Introduction}

The class of $\mathscr{L}_\infty$ hereditarily indecomposable (HI) spaces is perhaps the most interesting class of non-classical Banach spaces. This happens since in such a space $\mathfrak{X}$ conditional and unconditional structures strongly coexist. More precisely, in   $\mathfrak{X}$ there is no unconditional basic sequence and on the other hand $\mathfrak{X} = \overline{\cup_nF_n}$ with $(F_n)_n$ an increasing sequence of finite dimensional subspaces, each one $C$-isomorphic to $\ell_\infty^{\dim F_n}$. The latter yields that  $\mathfrak{X}$ admits Gordon-Lewis LUST \cite{GL}. It is an important open problem whether there exists a reflexive HI space with LUST. As a consequence of the above described peculiar structure, the scalar-plus-compact property is satisfied by several $\mathscr{L}_\infty$-spaces (e.g. \cite{AH},\cite{7a}). In this paper we present a new $\mathscr{L}_\infty$ HI space denoted $\X$. This is the first example of a $\mathscr{L}_\infty$ HI space not containing a reflexive subspace.
Moreover, every $T\in\mathcal{L}(\X)$ is of the form $\la I + K$ with $K$ a compact operator, and thus, this is the first example of a space without reflexive subspaces satisfying the scalar-plus-compact property.

In 1981 J. Bourgain \cite[Problem 4, page 46]{B} suggested   the class of $\mathscr{L}_\infty$-spaces as a possible subclass of Banach spaces where the problem ``$\ell_1$, $c_0$, or reflexive subspace''   could have a positive answer. This would be in line with a multitude of results hinting that $\mathscr{L}_\infty$-spaces exhibit highly canonical structure. For example, such spaces have the aforementioned Gordon-Lewis LUST and the dual of a separable $\mathscr{L}_\infty$-space is either isomorphic to $\ell_1$ or isomorphic to $\mathcal{M}[0,1]$ \cite{St}. Furthermore, as it was proved by H. P. Rosenthal in \cite{R}, whenever a $\mathscr{L}_\infty$-space embeds in a space with an unconditional basis then it is necessarily isomorphic to $c_0$. It follows from the work of D. Lewis-C. Stegall \cite{LS},\cite{St} and A. Pe\l czy\' nski \cite{P} that if the dual of a separable $\mathscr{L}_\infty$-space $\mathfrak{X}$ is non-separable, then $\ell_1$ is isomorphic to a subspace of $\mathfrak{X}$. If in addition $\mathfrak{X}^*$ is separable (i.e. $\mathfrak{X}^*\simeq\ell_1$) and $\mathfrak{X}$ does not contain a reflexive subspace, then $c_0$ appears as a strong candidate to be a subspace of $\mathfrak{X}$.

The aforementioned problem, in the general setting, was answered in 1994 by W. T. Gowers. More precisely, in  \cite{G},  the Gowers Tree space is presented, the first example of a Banach space not containing $\ell_1$, $c_0$, or a reflexive subspace. A systematic study of this type of spaces has appeared in \cite{AAT}. Gowers Tree space  and the spaces in \cite{AAT} satisfy a stronger property, namely every subspace has non-separable dual. Actually, in every subspace there exists a tree-basis similar to the basis of the classical James Tree space \cite{J}. This is an obstacle to any attempt to combine Gowers' norming set with the Bourgain-Delbaen techniques to obtain a $\mathscr{L}_\infty$ HI space with no reflexive subspace. Indeed, Gowers' norming set would enforce the dual of the space to be non-separable and since the space is $\mathscr{L}_\infty$, as we have mentioned before, the space $\ell_1$ would be isomorphic to a subspace of the space.

Motivated by the above, we recently introduced a new method of defining norming sets that, among others, yields HI Banach spaces with separable dual containing no reflexive subspaces. A Tsirelson version of this method and its consequences in a classical setting appeared in \cite{AM2}. It is worth pointing out that the new method leads to a unified approach for constructing HI spaces that are either reflexive or do not contain a reflexive subspace. Moreover, this is rather simpler than the initial method for constructing Gowers Tree spaces \cite{G},\cite{AAT}.

All known non-classical separable $\mathscr{L}_\infty$-spaces are Bourgain-Delbaen $\mathscr{L}_\infty$-spaces (BD-$\mathscr{L}_\infty$-spaces). This class of spaces was introduced by J. Bourgain and F. Delbaen \cite{BD} and they are defined as follows. A BD-$\mathscr{L}_\infty$-space is a subspace $\mathfrak{X}$ of $\ell_\infty(\Gamma)$, with $\Gamma$ a countable set. It is determined by a sequence $(\Ga_q,i_q)_q$ where $(\Ga_q)_q$ is an increasing sequence of finite subsets of $\Ga$ with $\cup_q\Ga_q = \Ga$ and $i_q:\ell_\infty(\Ga_q)\to\ell_\infty(\Ga)$, $q\in\N$ are uniformly bounded extension operators (i.e. $i_q(x)|_{\Ga_q} = x$) that are in addition compatible. This last property means that for $q<p$ and $x\in\ell_\infty(\Ga_q)$ we have $i_q(x) = i_p(i_q(x)|_{\Ga_p})$. For $q\in\N$ we set $\De_q = \Ga_q\setminus\Ga_{q-1}$ and for $\ga\in\De_q$ we set $d_\ga = i_q(e_\ga)$. Then, the BD-$\mathscr{L}_\infty$-space is defined to be $\mathfrak{X} = \overline{\langle\{d_\ga: \ga\in\Ga\}\rangle}$, as a
subspace of $\ell_\infty(\Ga)$. The sequence $(d_\ga)_{\ga\in\Ga}$ forms a Schauder basis for $\mathfrak{X}$, however it is usually more convenient to consider the finite dimensional decomposition (FDD) $(M_q)_q$ with $M_q = \langle\{d_\ga:\ga\in \De_q\}\rangle$. For an interval $E$ of $\N$, $P_E$ denotes the natural projection onto $E$ associated to the FDD $(M_q)_q$. As we mentioned above, this class of $\mathscr{L}_\infty$-spaces appeared for the first time in \cite{BD} as a specific class of $\mathscr{L}_\infty$-spaces. Recently, in \cite{AGM} it was shown that every separable $\mathscr{L}_\infty$-space is isomorphic to a BD-$\mathscr{L}_\infty$-space. Thus, BD-$\mathscr{L}_\infty$-spaces are the generic ones.

A second component in the Bourgain-Delbaen invention, which is equally important to the definition of  the spaces,  concerns the method of constructing the sequence $(i_q)_q$. It is defined inductively in a way to preserve the uniform bound of the the norms of the $i_q$'s.  Moreover, analyzing the initial spaces defined in \cite{BD} one can observe that the saturation of the structure is an inevitable ingredient. This is more transparent in the alternative definition of $\ell_p$-saturated $\mathscr{L}_\infty$-spaces in \cite{GPZ}. This explains why it is possible to combine BD-$\mathscr{L}_\infty$ structure with saturated norms resulting in $\mathscr{L}_\infty$ HI spaces. The relation of BD-$\mathscr{L}_\infty$-spaces with saturated norms was established for the first time by R. Haydon in \cite{H}.

Let us pass to the description of some features of the space $\X$. As we have already mentioned, we will use a new method of defining HI spaces, which we will combine with the Bourgain-Delbaen techniques in order to obtain a $\mathscr{L}_\infty$ HI space without reflexive subspaces. The new method requires a preexisting space that in the classical setting would be either the Tsirelson space \cite{T} or a mixed-Tsirelson space $T[(\mathscr{A}_{n_j},m_j^{-1})_j]$. The norming sets of the new spaces are defined to be subsets of the corresponding ones of the initial spaces. A typical and known example is Schlumprecht space $S[(\mathscr{A}_n,1/\sqrt{\log(n+1)})_n]$ \cite{S} that serves as a preexisting reflexive space with an unconditional basis for Gowers-Maurey space \cite{GM}. Furthermore, in \cite{AM2} the norming set $W$ is a subset of $W_T$, the norming set of Tsirelson space.

Attempting to adapt the above scheme to a $\mathscr{L}_\infty$ setting, we have to use an initial $\mathscr{L}_\infty$-space in which the obvious norming set is the set $\{e_\ga^*: \ga\in\Ga\}$. We then have to carefully select a subset of $\Ga$ that will define the space $\X$. This is the motivation behind introducing the self-determined subsets of a set $\Ga$, which are defined as follows.

\begin{dfn*}
Let $\mathfrak{X}$ be a BD-$\mathscr{L}_\infty$-subspace of $\ell_\infty(\Ga)$. A subset $\Ga'$ of $\Ga$ is self-determined if $\langle\{d_\ga^*: \ga\in\Ga'\}\rangle = \langle\{e_\ga^*: \ga\in\Ga'\}\rangle$, where $(d_\ga^*)_{\ga\in\Ga}$ denotes the sequence biorthogonal to the basis $(d_\ga)_{\ga\in\Ga}$ and for $\ga\in\Ga$, $e_\ga^*$ denotes the element $e_\ga$ of $\ell_1(\Ga)$ restricted on $\mathfrak{X}$.
\end{dfn*}
 The following holds.

\begin{prp*}
Let $\mathfrak{X}$ be a BD-$\mathscr{L}_\infty$-subspace of $\ell_\infty(\Ga)$ and $\Ga'$ be a self-determined subset of $\Ga$.
\begin{itemize}

\item[(i)] The space $Y = \overline{\langle\{d_\ga: \ga\in\Ga\setminus\Ga'\}\rangle}$ is a $\mathscr{L}_\infty$-space.

\item[(ii)] The quotient $\mathfrak{X}/Y$ is a $\mathscr{L}_\infty$-space.

\end{itemize}
\end{prp*}
The above proposition and a result from \cite{KL} yield the following.

\begin{thm*}
There is a continuum of $\mathscr{L}_\infty$-subspaces  $\{Y_\alpha:\alpha \in\mathfrak{c}\}$ of $\mathfrak{X}_{\mathrm{AH}}$, satisfying the following.
\begin{itemize}

\item[(i)] Each space $Y_\alpha$ has the scalar-plus-compact property and for every $\alpha \neq \beta$ every bounded linear operator $T:Y_\alpha\rightarrow Y_\beta$ is compact.

\item[(ii)] Each $X_\alpha = \mathfrak{X}_{\mathrm{AH}}/Y_\alpha$ is a hereditarily indecomposable space with the scalar-plus-compact property and for every $\alpha\neq\beta$ every bounded linear operator $T:X_\alpha\rightarrow X_\beta$ is compact.

\end{itemize}
\end{thm*}
The space $\mathfrak{X}_{\mathrm{AH}}$ above, is the $\mathscr{L}_\infty$ HI space from \cite{AH}. Another consequence of self-determined sets is an intriguing result that displays the complete divergence between the structure of a $\mathscr{L}_\infty$-space and its quotients. This contrasts corresponding results concerning classical $\mathscr{L}_\infty$-spaces \cite{JZ}.
\begin{thm*}
There exist $\mathscr{L}_\infty$ Banach spaces $X_1, X_2, X_3$ with separable dual so that $X_2$ is a quotient of $X_1$ and $X_3$ is a quotient of $X_2$ and moreover $X_1$ and $X_3$ are reflexive saturated whereas $X_2$ contains no reflexive subspaces.
\end{thm*}

In the present construction, the preexisting space is the space $\mathfrak{B}_{\mathrm{mT}}$ from \cite{AH}. This is a BD-$\mathscr{L}_\infty$-space, saturated by  reflexive subspaces that have an unconditional basis. The norming set of the space $\X$ will be a self-determined subset $\Ga'$ of the set $\Ga$ defining the space $\mathfrak{B}_{\mathrm{mT}}$ as a subspace of $\ell_\infty(\Ga)$. As in all Bourgain-Delbaen constructions, to each $\ga\in\Ga$ with $\ga\in\De_{q+1}$ we associate a linear functional $c_\ga^*:\ell_\infty(\Ga_q)\to\R$ such that $e^*_\gamma=c_\gamma^*+d_\gamma^*$. In the case of $\X$, the functional $c^*_\gamma$ is defined as
$$c_\ga^* =   \frac{1}{m_j}b^*, \ \ \text{or } \ \ c_\ga^* = e_\xi^* + \frac{1}{m_j}b^*,$$
where $\xi\in\De_p\cap\Ga'$, $p<q$  and
$$b^* = \frac{1}{n}\left(\e_1 e^*_{\zeta_1}\circ P_{E_1}+ \cdots + \e_n e_{\zeta_n}^*\circ P_{E_n}\right)$$
with $\zeta_1,\ldots,\zeta_n\in \Ga'\cap(\Ga_q\setminus\Ga_p)$, $p<E_1<\cdots<E_n\leqslant q$ and $\e_1,\ldots,\e_n\in\{-1,1\}$.  If $\xi $ does not exist, then $\zeta_1,\ldots,\zeta_n\in \Ga'\cap \Ga_q $.

The functional $b^*$ is a special type of average, called an  \ac-average. The latter are defined by a countable tree $\mathcal{U}$ with elements $\{(\ga_k,x_k)\}_{k=1}^n$ so that $\ga_k\in\Ga$ and $x_k$ is a  block vector in $\mathfrak{B}_{\mathrm{mT}}$ with respect to its basis $(d_\ga)_{\ga\in\Ga}$. The use of the tree $\mathcal{U}$ in the definition of  \ac-averages explains the necessity of the preexisting space $\mathfrak{B}_{\mathrm{mT}}$. Furthermore, since $\Gamma'$ is a self determined subset of $\Gamma$, the space $\X$ is a quotient of $\mathfrak{B}_{\mathrm{mT}}$.

Note that in the definition of $\Ga'$ we use saturation under constraints, which has occurred in earlier papers (e.g. \cite{AM1}, \cite{ABM}). The version appearing here is similar to the one used in \cite{AM2}. The difference from \cite{AM2} is that here we deal with families $(\mathscr{A}_{n_j})_j$ instead of $(\mathcal{S}_n)_n$, which makes the definitions and the proofs easier. It is also worth pointing out that, as in \cite{AM2}, the conditional structure  of the space $\X$ is imposed by certain \ac-averages and not by special sequences $(\ga_k)_{k=1}^n$.

The space $\X$ satisfies the scalar-plus-compact property. In the space $\mathfrak{X}_{\mathrm{AH}}$ from \cite{AH} the same result is proved using LUST. In the case of $\X$ the proof is more involved. This is due to the fact that we apply saturation under constraints. Actually, we combine the LUST of $\X$ and the fact $\X^*\simeq\ell_1$. Another property of $\X$ is that every subspace fails the point of continuity property (PCP). Thus, $\X$ answers in a strong sense a problem posed by J. Bourgain \cite[Problem 3, page 46]{B}. We mention that a $\mathscr{L}_\infty$-space without PCP and not containing $c_0$ could also be obtained by the results in \cite{FOS}.

We close the introduction by mentioning a problem attributed to H. P. Rosenthal. The problem in question is the following.
\begin{prb*}
Let $\mathfrak{X}$ be a $\mathscr{L}_\infty$ saturated Banach space. Does $\mathfrak{X}$ contain $c_0$ isomorphically?
\end{prb*}

Note that if $\mathfrak{X}$ is $\mathscr{L}_\infty$ saturated and does not contain $c_0$ then it does not contain an unconditional basic sequence. This follows from James' classical characterization of reflexivity for spaces with an unconditional basis. Indeed, let us assume that $\mathfrak{X}$ contains a subspace $Y$ with an unconditional basis. Then $Y$ is $\mathscr{L}_\infty$ saturated i.e. it is not reflexive. For the same reason $Y$ cannot contain $\ell_1$  and by assumption it does not contain $c_0$, i.e. $Y$ is reflexive which is absurd. Therefore, by Gowers' dichotomy \cite{G2} any $\mathscr{L}_\infty$ saturated space $\mathfrak{X}$ not containing $c_0$ is HI saturated. If a space $\mathfrak{X}$ answers Rosenthal's problem negatively, then it has to be saturated with HI spaces with LUST and probably with the scalar-plus-compact property. We show that $\X$ is not $\mathscr{L}_\infty$ saturated. However, we believe that  the techniques deployed in the present paper are a step   towards the solution of Rosenthal's problem.

\section{The self-determined sets}

In this preparatory section we introduce the self-determined subsets of the norming set $\Gamma$ of a BD-$\mathscr{L}_\infty$-space  $\mathfrak{X}$.  It is shown that the self-determined sets are able to provide  $\mathscr{L}_\infty$  subspaces and quotients of a given BD-$\mathscr{L}_\infty$-space. Also, they are a key ingredient for the definition of the space  $\X$. We start by recalling the definition of the BD-$\mathscr{L}_\infty$-spaces given in  \cite{AGM}. We remind that every separable $\mathscr{L}_\infty$ space is isomorphic to a BD-$\mathscr{L}_\infty$-space,  (\cite[Theorem 3.6]{AGM}).

\begin{ntt}
For $\Gamma_1$, $\Gamma$ sets with $\Gamma_1\subseteq \Gamma$, we denote by $r: \ell^\infty(\Gamma)\rightarrow \ell^\infty (\Gamma_1)$ the natural restriction operator. An operator $i: \ell^\infty (\Gamma_1)\rightarrow \ell^\infty (\Gamma)$ is an extension operator if $r\circ i$ is the identity operator of $\ell^\infty (\Gamma_1)$. 
Also, if $(\Ga_q)_{q}$ is a strictly increasing sequence of non-empty sets and $\Gamma = \cup_q\Ga_q$, a sequence of extension operators $(i_q)_{q}$, with $i_q: \ell_\infty(\Ga_q) \rightarrow \ell_\infty(\Ga)$ for all $q\inn$, will be called compatible, if for every $p,q \inn$ with $p < q$, $i_p = i_q\circ r_q\circ i_p$, where $r_q$ denotes the restriction onto $\Ga_p$.
\end{ntt}

\begin{dfn}\label{definition BD}
Let $(\Ga_q)_{q=1}^\infty$ be a strictly increasing sequence of non-empty finite sets, $\Ga = \cup_q\Ga_q$ and $(i_q)_{q=1}^\infty$, with $i_q: \ell_\infty(\Ga_q) \rightarrow \ell_\infty(\Ga)$ for all $q\inn$,  such that $C = \sup_q\|i_q\|$ is finite. Define $\De_1 = \Ga_1$, $\De_{q+1} = \Ga_{q+1}\setminus\Ga_{q}$ for $q\inn$ and for every $\ga\in\Ga$ we define $d_\ga$, a vector in $\ell_\infty(\Ga)$, as follows: if $\ga\in\De_q$ for some $q\inn$, then $d_\ga = i_q(e_\ga)$. The closed linear span of the set $\{ d_\ga:\ga\in\Ga \}$ will be denoted by $\Xbd$ and called a Bourgain-Delbaen space.
\end{dfn}

If for all $q\inn$ we define $M_q = \langle\{d_\ga:\ga\in\De_q\}\rangle$, then $(M_q)_q$ forms a Finite Dimensional Decomposition (FDD) for the space $\Xbd$ and for every interval $E$ of $\N$  we denote by $P_E$ the projection associated to this FDD and $E$. For every $\ga\in\Ga$ we denote by $e_\ga^*:\mathfrak{X}_{(\Ga_q,i_q)_q}\rightarrow\mathbb{R}$ the evaluation functional on the $\ga$'th coordinate, defined on $\ell_\infty(\Ga)$ and then restricted to the subspace $\mathfrak{X}_{(\Ga_q,i_q)_q}$. Moreover, for every $\ga\in\Ga$ we define two specific linear functionals $c_\ga^*$ and $d_\ga^*$ so that $e_\ga^* = c_\ga^* + d_\ga^*$. For the precise definition see \cite[Definition 2.14]{AGM}. We summarize some properties of these functionals. Their proofs can be found in \cite[Lemma 2.16 and Proposition 2.17]{AGM}.

\begin{prp}\label{functional properties}
Let $\Xbd$ be a Bourgain-Delbaen space. The following hold.
\begin{itemize}

 \item[(i)] The sequence $(e_\ga^*)_{\ga\in\Ga}$ is equivalent to the unit vector basis of $\ell_1(\Ga)$.

 \item[(ii)] The functionals $(d_\ga^*)_{\ga\in\Ga}$ are biorthogonal to the vectors $(d_\ga)_{\ga\in\Ga}$.

 \item[(iii)] For $q\inn\cup\{0\}$, $\{c_\ga^*:\ga\in\De_{q+1}\} \subset \langle\{e_\ga^*:\ga\in\Ga_q\}\rangle = \langle\{d_\ga^*:\ga\in\Ga_q\}\rangle$.

 \item[(iv)] If the FDD $(M_q)_q$ is shrinking, the closed linear span of the functionals $(d_\ga^*)_{\ga\in\Ga}$ is $\mathfrak{X}_{(\Ga_q,i_q)_q}^*$ and hence, $\mathfrak{X}_{(\Ga_q,i_q)_q}^*$ is isomorphic to $\ell_1$.
 \end{itemize}
\end{prp}

\subsection{Self-determined subsets of $\Ga$}
We define and study self-determined subsets $\Ga'$ of $\Ga$, which define quotients of $\Xbd$, which are Bourgain-Delbaen spaces as well.

\begin{dfn}
Let $\Xbd$ be a Bourgain-Delbaen space. An infinite subset $\Ga'$ of $\Ga$ will be called self-determined, if for every $\ga\in\Ga'$ the functional $d_\ga^*$ is in the linear span of $\{e_\ga^*:\ga\in\Ga'\}$.
\end{dfn}

\begin{prp}\label{self-determinacy is all those nice things}
Let $\Xbd$ be a Bourgain-Delbaen space, $\Ga'$ be an infinite subset of $\Ga$, for all $q$ set $\Ga_q' =\Ga'\cap\Ga_q$, $\De_q' = \Ga'\cap\De_q$ and $q_0 = \min\{q: \Ga_q'\neq\varnothing\}$. The following assertions are equivalent.
\begin{itemize}

\item[(a)] The set $\Ga'$ is self-determined.

\item[(b)] For all $q\geqslant q_0$, $\langle\{d_\ga^*:\ga\in\Ga_q'\}\rangle = \langle\{e_\ga^*:\ga\in\Ga_q'\}\rangle$.

\item[(c)] We have $\langle\{d_\ga^*:\ga\in\Ga_{q_0}'\}\rangle = \langle\{e_\ga^*:\ga\in\Ga_{q_0}'\}\rangle$ and for all $q\geqslant q_0$, $\langle\{c_\ga^*:\ga\in\De_{q+1}'\}\rangle \subset \langle\{e_\ga^*\circ P_E:\ga\in\Ga_q', E\subset\N\cup\{0\}\}\rangle$.

\item[(d)] For all $\ga\in\Ga\setminus\Ga'$ and $\eta\in\Ga'$, $e_\eta^*(d_\ga) = 0$.

\item[(e)] For all $\ga\in\Ga\setminus\Ga'$ and $\eta\in\Ga'$, $c_\eta^*(d_\ga) = 0$.

\end{itemize}
\end{prp}

\begin{proof}
We add an auxiliary assertion in order to obtain the equivalence.
\begin{itemize}
\item[(c')] We have $\langle\{d_\ga^*:\ga\in\Ga_{q_0}'\}\rangle = \langle\{e_\ga^*:\ga\in\Ga_{q_0}'\}\rangle$ and for all $q\geqslant q_0$, $\langle\{c_\ga^*:\ga\in\De_{q+1}'\}\rangle \subset \langle\{d_\ga^*:\ga\in\Ga_q'\}\rangle$.
\end{itemize}
The main facts we shall use are (ii) and (iii) from Proposition \ref{functional properties}, as well as $e_\ga^* = c_\ga^* + d_\ga^*$ for all $\ga\in\Ga$. The assertions (a)$\Leftrightarrow$(b), (b)$\Rightarrow$(c'), (c')$\Rightarrow$(e) and (d)$\Leftrightarrow$(e) are very easy to prove. To see (c')$\Rightarrow$(c) recall that from \cite[Remark 2.15]{AGM}, for each $p\inn$ and $\ga\in\De_p$, $d_\ga^* = e_\ga^*\circ P_{\{p\}}$, while (c)$\Rightarrow$(b) is proved by induction on $q$. An argument involving kernels of linear functionals yields (e)$\Rightarrow$(c'). Drawing a diagram will convince the reader that the proof is complete.
\end{proof}

\begin{ntt}\label{some prime notation}
Given a Bourgain-Delbaen space $\Xbd$ as well as a self-determined subset $\Ga'$ of $\Ga$, we denote by
\begin{itemize}

\item[(i)] $R$ the restriction onto $\Ga'$,

\item[(ii)] $\Ga_q' = \Ga'\cap\Ga_q$ and $\De_q' = \Ga'\cap\De_q$ for all $q\inn$,

\item[(ii')] $\Ga_q'' = \Ga_q\setminus\Ga_q'$ and $\De_q'' = \De_q\setminus\De_q'$ for all $q\inn$

\item[(iii)] $S = \{q\inn\cup\{0\}: \De_q'  \neq\varnothing\} = \{q_0 < q_1 <\cdots < q_s \cdots\}$,

\item[(iv)] for all $s\inn\cup\{0\}$, $r_{q_s}'$ the restriction onto $\Ga_{q_s}'$ and

\item[(iv)] for $s\inn\cup\{0\}$, $i_{q_s}':\ell_\infty(\Ga_{q_s}')\rightarrow\ell_\infty(\Ga')$ with $i_{q_s}'(x) = R(i_{q_s}(x))$, where we naturally identify $x$ with a vector in $\ell_\infty(\Ga_{q_s})$.

\end{itemize}
Observe that $(\Ga_{q_s}')_{s=1}^\infty$ is a strictly increasing sequence of finite sets, whose union is $\Ga'$ and that $(i_{q_s}')_{q=1}^\infty$ is a uniformly bounded sequence of extension operators, in particular $\sup_s\|i_{q_s}'\| \leqslant \sup_q\|i_q\|$.
\end{ntt}

For the rest of this section we follow the above notation.

\begin{prp}\label{complement gives subspace}
Let $\Xbd$ be a Bourgain-Delbaen space and $\Ga'$ be a self-determined subset of $\Ga$. Then for every $q\inn$, we have that $i_q[\ell_\infty(\Ga_q'')] = \langle\{d_\ga:\ga\in\Ga_q''\}\rangle$, where we naturally identify $\ell_\infty(\Ga_q'')$ with a subspace of $\ell_\infty(\Ga_q)$. In particular, if we denote by $Y$ the closed linear span of $\{d_\ga:\ga\in\Ga\setminus\Ga'\}$, then $Y$ is a $\mathscr{L}_\infty$-space.
\end{prp}

\begin{proof}
Fix $q\inn$. If $\De_q'' = \varnothing$ then there is nothing to prove. If this is not the case, we will show that $d_\ga\in i_q[\ell_\infty(\Ga_q'')]$ for all $\ga\in\Ga_q''$ which, due to dimensional reasons, yields the desired result. The compatibility property of the operators implies $d_\ga = i_q(\sum_{\eta\in\Ga_q}e_\eta^*(d_\ga)e_\eta)$ whereas Proposition \ref{self-determinacy is all those nice things} (d) yields $d_\ga = i_q(\sum_{\eta\in\Ga_q''}e_\eta^*(d_\ga)e_\eta)$. The second part of this Proposition follows from the fact that for all $q$, $i_q$ is a $C$-isomorphism, where $C = \sup_q\|i_q\|$.
\end{proof}

\begin{lem}\label{commute}
Let $\Xbd$ be a Bourgain-Delbaen space and $\Ga'$ be a self-determined subset of $\Ga$. Then for every $s\inn$ and $x\in\ell_\infty(\Ga_{q_s})$ we have $R(i_{q_s}(x)) = i_{q_s}'(r_{q_s}'(x))$.
\end{lem}

\begin{proof}
Note that $x - r_{q_s}'(x) = \sum_{\eta\in\Ga_{q_s}''}e_\eta^*(x)e_\eta$ and hence $i_{q_s}'(r_{q_s}'(x)) = R(i_{q_s}(r_{q_s}'(x))) = R(i_{q_s}((x))) - R(i_{q_s}(\sum_{\eta\in\Ga_{q_s}''}e_\eta^*(x)e_\eta))$. We will show that $R(i_{q_s}(e_\eta)) = 0$ for all $\eta\in\Ga_{q_s}''$, which clearly yields the desired result. By Proposition \ref{complement gives subspace},  $i_{q_s}(e_\eta)$ is in $\langle\{d_\ga:\ga\in\Ga_q''\}\rangle$ and by Proposition \ref{self-determinacy is all those nice things} (d) the conclusion follows.
\end{proof}

\begin{prp}\label{self-determined defines bd space}
Let $\Xbd$ be a Bourgain-Delbaen space and $\Ga'$ be a self-determined subset of $\Ga$. Then the sequence $(i_{q_s}')_{s=1}^\infty$ is compatible, hence it defines a Bourgain-Delbaen space $\Xq$.
\end{prp}

\begin{proof}
Fix $s$, $t$ in $\N$ with $s < t$ and $x\in\ell_\infty(\Ga_{q_s}')$, we will show that $i_{q_s}'(x) = i_{q_t}'(r_{q_t}'(i_{q_s}'(x)))$. Define $y = r_{q_t}'(i_{q_s}'(x))$ and observe that $r_{q_t}'\circ R = r_{q_t}'\circ r_{q_t}$, which in conjunction with $i_{q_s}' = R\circ i_{q_s}$, yields $y = r_{q_t}'(r_{q_t}(i_{q_s}(x)))$. Applying Lemma \ref{commute} and using the compatibility of $(i_q)_q$ we obtain
\begin{eqnarray*}
i_{q_t}'\left(r_{q_t}'\left(i_{q_s}'(x)\right)\right) &=& i_{q_t}'\left(r_{q_t}'\left(r_{q_t}\left(i_{q_s}(x)\right)\right)\right)
=  R\left(i_{i_{q_t}}\left(r_{q_t}\left(i_{q_s}(x)\right)\right)\right)\\ &=& R(i_{q_s}(x)) = i_{q_s}'(x).
\end{eqnarray*}
Since $\sup_{s}\|i_{q_s}'\| \leqslant \sup_q\|i_q\|$, the proof is complete.
\end{proof}

\begin{ntt}\label{some extra-prime notation}
Given a Bourgain-Delbaen space $\Xbd$ as well as a self-determined subset $\Ga'$ of $\Ga$, we denote by $(d_\ga')_{\ga\in\Ga'}$ the vectors that span the space $\Xq$. Moreover, for $\ga\in\Ga'$ we denote by $d_\ga^{\prime*}$ and $c_\ga^{\prime*}$ the corresponding functionals from \cite[Definition 2.14]{AGM}, whereas for the evaluation functionals on $\gamma$'th coordinate we retain the symbol $e_\ga^*$. We also denote by $(M_s')_{s=0}^\infty$ the FDD of $\Xq$ as defined in \cite[Proposition 2.8]{AGM} and by $P_E'$ the corresponding projections.
\end{ntt}

\begin{rmk}\label{quotient map}
Note that Lemma \ref{commute} easily implies that for every $\ga\in\Ga$, $R(d_\ga)$ is in $\Xq$. Hence, $R:\Xbd\rightarrow\Xq$ is a well defined linear operator of norm at most one. Moreover, the following hold:
\begin{itemize}

\item[(i)] for all $\ga\in\Ga'$, $R(d_\ga) = d_\ga'$,

\item[(ii)] for all $\ga\in\Ga\setminus\Ga'$, $R(d_\ga) = 0$,

\item[(iii)] for all $\ga\in\Ga'$, $R^{*}(e_\ga^*) = e_\ga^*$,

\item[(iv)] for all $\ga\in\Ga'$, $R^{*}(d_\ga^{\prime*}) = d_\ga^*$ and

\item[(v)] for all $\ga\in\Ga'$, $R^{*}(c_\ga^{\prime*}) = c_\ga^*$.

\end{itemize}
The first assertion is easily implied by Lemma \ref{commute} while the second one clearly follows from Proposition \ref{self-determinacy is all those nice things} (d). The third assertion is an easy consequence of the definition of $R$, the fourth follows from the first two ones while the last one follows from (iii) and (iv).
\end{rmk}

\begin{prp}\label{quotient onto prime space}
Let $\Xbd$ be a Bourgain-Delbaen space, $\Ga'$ be a self-determined subset of $\Ga$ and $Y$ be the closed linear span of $\{d_\ga:\ga\in\Ga\setminus\Ga'\}$. Then $R$ is onto $\Xq$ and its kernel is the space $Y$. Hence, $\Xbd/Y$ is isomorphic to $\Xq$.
\end{prp}

\begin{proof}
To conclude that $R$ is onto, it suffices to show that the closure of $R[A]$ contains $B$, where $A = \{x\in\Xbd: \|x\| \leqslant \sup_q\|i_q\|\}$ and $B$ is the unit ball of $\Xq$. To this end, let $x$ be in the linear span of $\{d_\ga':\ga\in\Ga'\}$, with $\|x\|\leqslant 1$. Then there is $s\inn\cup\{0\}$ and $x\in\ell_\infty(\Ga_{q_s}')$ so that $x = i_{q_s}'(y)$. If we define $z = i_{q_s}(y)$, then $\|z\| \leqslant \|i_{q_s}\|$ and by Lemma \ref{commute}, $R(z) = x$. To conclude the proof, observe that Remark \ref{quotient map}, in conjunction that $((d_\ga)_{\ga\in\De_q})_q$ and $((d_\ga')_{\ga\in\De_{q_s}'})_s$ are be Schauder bases for $\Xbd$ and $\Xq$ respectively, yields that $\ker R = Y$.
\end{proof}

The Proposition below and the remark following it, state that if $\ga\in\Ga'$ and the action of the extension function $c_\ga^*$ is understood, then the action of $c_\ga^{\prime*}$ is understood as well. This is useful for determining the evaluation analysis of the coordinate $\ga$.

\begin{prp}\label{analysis invariant}
Let $\Xbd$ be a Bourgain-Delbaen space and $\Ga'$ be a self-determined subset of $\Ga$. Let moreover $\ga$ be in $\Ga'$ and assume that there are a finite subset $F$ of $\Ga'$, scalars $(\la_\eta)_{\eta\in F}$ and intervals  $(E_\eta)_{\eta\in F}$ of $\N$ so that $c_\ga^* = \sum_{\eta\in F}\la_\eta e_\eta^*\circ P_{E_\eta}$. Then $c_\ga^{\prime*} = \sum_{\eta\in F}\la_\eta e_\eta^*\circ P_{E_\eta'}'$, where  for $\eta\in F$, $E_\eta' = \{s\inn\cup\{0\}: q_s\in E_\eta\}$.
\end{prp}

\begin{proof}
We will show that for all $\xi$ in $\Ga'$, $c_\ga^{\prime*}(d_\xi') = \sum_{\eta\in F}\la_\eta e_\eta^*\circ P_{E_\eta'}'(d_\xi')$. Fix $\xi\in\Ga'$ with $\xi\in\De_{q_s}'$.  Remark \ref{quotient map} (i) and (v) yield $c_\ga^{\prime*}(d_\xi') = c_\ga^*(d_\xi)$ and hence, setting $F' = \{\eta\in F: q_s\in E_\eta\}$, we obtain
\begin{eqnarray*}
c_\ga^{\prime*}(d_\xi')&=& \sum_{\eta\in F}\la_\eta e_\eta^*\circ P_{E_\eta}(d_\xi) = \sum_{\eta\in F'}\la_\eta e_\eta^*(d_\xi)\\ &=& \sum_{\eta\in F'}\la_\eta e_\eta^*(d_\xi') = \sum_{\eta\in F}\la_\eta e_\eta^*\circ P_{E_\eta'}'(d_\xi').
\end{eqnarray*}
\end{proof}

\begin{rmk}\label{bars cancel out and stuff}
The above argument actually yields that if $\ga$, $\xi$ are in $\Ga'$ and $E$ is an interval of $\N$, then $e_\ga^*\circ P_E(d_\xi) = e_\ga^*\circ P_{E'}'(d_\xi')$.
\end{rmk}

\begin{rmk}
In \cite{FOS} the authors introduce a method of embedding a separable Banach space $X$ into a Bourgain-Delbaen space $Z$ satisfying certain properties, in particular if $X$ has separable dual then $Z^*$ is isomorphic to $\ell_1$. This method comprises of two steps. In the first one, $X$ is embedded into a Bourgain-Delbaen space $Y = \Xbd$ whereas in the second one, the space $\Xbd$ is ``augmented'' to obtain a space $Z = \mathfrak{X}_{(\bar{\Ga}_q,\bar{i}_q)}$, with $\bar{\Ga} = \cup_q\bar{\Ga}_q$ a suitable superset of $\Ga = \cup_q\Ga_q$ so that a copy of $X$ is naturally preserved in $Z$. As it is stated in that paper, the restriction operator $R$ onto $\Ga$ maps elements of $Z$ to elements of $Y$. We observe that $\Ga$ is a self determined subset of $\bar{\Ga}$ and hence, $R:Z\rightarrow Y$ is a quotient map which moreover preserves a copy of $X$. In \cite{7a} it is shown that if $X$ is super-reflexive, $Z$ can be chosen to satisfy the scalar-plus-compact property. As we will also comment later,
 the Argyros-Haydon space from \cite{AH} is a quotient of the mixed-Tsirelson Bourgain-Delbaen space $\BmT$ defined in that paper.
\end{rmk}

\section{The definition of the space $\X$}

In this section we define the space $\X$, combining the method from \cite{AH} with that form \cite{AM2}. We start by recalling the definition of the Bourgain-Delbaen mixed-Tsirelson space $\mathfrak{B}_{\mathrm{mT}}$ from \cite{AH}, in fact a slight variation of it, and then we define the space $\X$ as a quotient of $\mathfrak{B}_{\mathrm{mT}}$ by selecting an appropriate self-determined subset of the set $\bar \Ga$ associated to $\mathfrak{B}_{\mathrm{mT}}$. We follow the notation from both papers \cite{AH} and \cite{AM2} and when they come in conflict, we shall use the one from \cite{AH}.

\subsection{The space $\BmT$}
In \cite{AH} a Bourgain-Delbaen space $\mathfrak{B}_{\mathrm{mT}}$ is presented which is based on a mixed-Tsirelson space. We slightly modify this space but still denote it by $\mathfrak{B}_{\mathrm{mT}}$. This construction can, by now, be considered standard and therefore we do not include all of the details. We shall use notation such as $\bar{\Ga}$, $\bar{\Ga}_q$, $\bar{i}_q$, $\bar{c}_\ga^*$, $\bar{d}_\ga^*$, $\bar{d}_\ga$, $(\bar{M}_q)_q$ and $\bar{P}_E$ to refer to the corresponding components of the space $\mathfrak{B}_{\mathrm{mT}}$, as we reserve the casual notation (i.e. $\Ga_q$, $c_\ga^*$ etc) for the space $\X$ that we will define in the sequel. We start with a sequence of pairs of natural numbers $(m_j,n_j)_{j=1}^\infty$ satisfying \cite[Assumption 2.3, page 5]{AH}. We shall use, without mentioning it, properties of this sequence such as $\sum_j1/m_j < 1/3$ and $\sum_{j>i}1/m_j < 1/m_i$, however we make the additional assumption that $m_1\geqslant 8$, which is not needed in \cite{AH}.

\begin{prp}\label{mtbd}
There exists a Bourgain-Delbaen space $\mathfrak{B}_{\mathrm{mT}} = \mathfrak{X}_{(\bar{\Ga}_q,\bar{i}_q)}$, with $\sup_q\|\bar{i}_q\|\leqslant 2$, so that $\bar{\De}_1 = \{0\}$ and for $q\inn$
\begin{equation*}
\begin{split}
\bar{\De}_{q+1} =& \bigcup_{j=0}^{q+1} \left\{(q+1, m_j^{-1}, b^*): b^*\in B_{0,n}\right\}\\
&\begin{split}\cup\bigcup_{p=0}^{q-1}\bigcup_{j=0}^p \{&(q+1, \xi, m_j^{-1}, b^*): \xi\in\bar{\De}_p, \we(\xi) = m_j^{-1}, \ag(\xi) < n_j,\\
& b^*\in B_{p,q}\}
\end{split}
\end{split}
\end{equation*}
where for each $0\leqslant p < q$, $B_{p,q}$ is the set of all linear combinations
\begin{equation}\label{some small but needed change in the definition of the dense subsets}
b^* = \sum_{\eta\in\bar{\Ga}_q\setminus\bar{\Ga}_p}\la_\eta e_\eta^*\circ \bar{P}_{E_\eta}
\end{equation}
where $\sum_{\eta\in\bar{\Ga}_q\setminus\bar{\Ga}_p}|\la_\eta|\leqslant 1$, each $\la_\eta$ is a rational number with denominator dividing the quantity $\mathcal{N}_{q+1} = (2^q\#\bar{\Ga}_q)!$ and each $E_\eta$ is an interval of $(p,q]$. For each $\ga\in \bar{\De}_{q+1}$, $\ra(\ga) = q+1$ and if $\ga = (q+1, m_j^{-1}, b^*)$, $\ag(\ga) = 1$, $\we(\ga) = m_j^{-1}$ and
\begin{subequations}
\begin{equation}\label{age one bar}
\bar{c}_\ga^* = \frac{1}{m_j}b^*
\end{equation}
whereas if $\ga = (q+1, \xi, m_j^{-1}, b^*)$, then $\ag(\ga) = \ag(\xi) + 1$, $\we(\ga) = m_j^{-1}$ and
\begin{equation}\label{older age bar}
\bar{c}_\ga^* = e_\xi^* + \frac{1}{m_j}b^*.
\end{equation}
\end{subequations}
\end{prp}

\begin{rmk}\label{we need more convex combinations, MORE}
Although the definition of the space $\mathfrak{B}_{\mathrm{mT}}$ is formulated slightly differently than in \cite{AH}, the only actual difference lies in the sets $B_{p,q}$, namely the $E_\eta$'s that appear in \eqref{some small but needed change in the definition of the dense subsets} are only allowed to be the interval $(p,q]$ in \cite[Section 4]{AH}.
\end{rmk}

\subsection{Constraints in the setting of the Bourgain-Delbaen construction method}

We adapt some notation used in papers such as \cite{ABM}, \cite{AM2} to the setting of our construction. We remind that such constraints have also been used in \cite{AGM}.
\begin{itemize}

\item[(i)] Let $\Xi$ be a subset of $\bar{\Ga}$. A functional $b^*$ in $\BmT^*$ is called an $\al$-average of $\Xi$ of size $s(b^*) = n$, if there exist $1\leqslant d\leqslant n$, successive intervals $(E_i)_{i=1}^d$ of $\N$, signs $(\e_i)_{i=1}^d$ in $\{-1,1\}$ and $(\ga_i)_{i=1}^d$ in $\Xi$ so that
$$b^* = \frac{1}{n}\sum_{i=1}^d\e_ie_{\ga_i}^*\circ \bar{P}_{E_i}.$$
Observe that if $p < \min E_1$, $\max E_d \leqslant q$ and $n \leqslant (2^q\#\bar{\Ga}_q)!$, then $b^*\in B_{p,q}$.

\item[(ii)] A (finite or infinite) sequence of $\al$-averages $(b^*_k)_k$ of $\Xi$ is called very fast growing, if there are non-negative integers $0\leqslant p_1 < q_1 < p_2 < q_2 < \cdots$ so that $b_k^*\in B_{p_k,q_k}$ for $k = 1,2,\ldots$ and $s(b_k^*) \geqslant \mathcal{N}_{q_{k-1}+1}$ for $k > 1$.

\end{itemize}
For the definition of $\mathcal{N}_q$ see below \eqref{some small but needed change in the definition of the dense subsets}. Note that a subsequence of a very fast growing sequence is itself very fast growing.

\begin{rmk}
\label{sizes indeed increase}
The above definition implies that if $(b_k^*)_k$ is very fast growing, then $s(b_k^*)<s(b_{k+1}^*)$ for all $k$.
\end{rmk}

\subsection{The tree of special sequences}\label{subsection tree}
We denote by $\mathcal{Q}$ the set of all finite sequences of pairs $\{(\ga_1, x_1),\ldots,(\ga_k, x_k)\}$ satisfying the following:
\begin{itemize}

\item[(i)] $\ga_i\in\bar{\Ga}$ with $\ra(\ga_i)\geqslant \min \ran x_i$ for $i=1,\ldots,k$ and

\item[(ii)] the $x_1,\ldots,x_k$ are finite linear combinations of $(\bar{d}_\ga)_{\ga\in\bar{\Ga}}$ with rational coefficients, which are successive with respect to the FDD $(\bar{M}_q)_q$.

\end{itemize}
We choose  a one-to-one function $\sigma:\mathcal{Q}\rightarrow\N$, called the coding function, so that for every $\{(\ga_1, x_1),\ldots,(\ga_k, x_k)\}\in\mathcal{Q}$
\begin{equation}\label{coding growth}
\sigma\left(\left\{\left(\ga_1, x_1\right),\ldots,\left(\ga_k, x_k\right)\right\}\right) > \we(\ga_k)^{-1}\max \supp x_k
\end{equation}
where the support $x_k$ is considered with respect to the FDD $(\bar{M}_q)_q$.

A finite sequence $\{(\ga_k, x_k)\}_{k=1}^d\in\mathcal{Q}$ is called a special sequence if:
\begin{itemize}

\item[(i)] $\we(\ga_1) = m_1^{-1}$ and

\item[(ii)] if $d\geqslant 2$ then $\we(\ga_k) = m_{\sigma((\ga_1, x_1),\ldots,(\ga_{k-1}, x_{k-1}))}^{-1}$ for $k=2,\ldots,d$.

\end{itemize}
We note by $\mathcal{U}$ the tree of all special sequences, endowed with the natural ordering ``$\sqsubseteq$'' of initial segments.

\begin{rmk}\label{special weights decreasing}
Note that if $\{(\ga_k, x_k)\}_{k=1}^d$ is a special sequence, then by \eqref{coding growth} $\we(\ga_1) > \cdots > \we(\ga_d)$.
\end{rmk}

\begin{dfn}\label{incomparable naturals}
We say that two distinct natural numbers $i$, $j \geqslant 2$  are incomparable if one of the following holds:
\begin{itemize}

\item[(i)] neither $i$ nor $j$ is in $\sigma(\mathcal{Q})$ or

\item[(ii)] both $i$ and $j$ are in $\sigma(\mathcal{Q})$ and $\sigma^{-1}(i)$, $\sigma^{-1}(j)$ are incomparable in the ordering of $\mathcal{U}$.

\end{itemize}
\end{dfn}

\subsection{The \ac-averages}
In a similar manner as in \cite{AM2}, we define specific types of averages, based on the tree $\mathcal{U}$ and the notion of comparability of natural numbers from Definition \ref{incomparable naturals}.

\begin{dfn}\label{def averages}
Let $\Xi$ be a subset of $\bar{\Ga}$, $d\inn$, $E_1<\cdots<E_d$ be intervals of $\N$ and $\ga_i\in\Xi$ with $\ra(\ga_i)\geqslant \min E_i$ for $i=1,\ldots, d$  and $(\we(\ga_i))_{i=1}^d$ is strictly decreasing.

\begin{itemize}

\item[(i)] The sequence of pairs $(\ga_i, E_i)_{i=1}^d$ is called incomparable, if choosing $j_i$ so that $\we(\ga_i) = m_{j_i}^{-1}$, then the natural numbers $j_i$, $i=1,\ldots,d$ are pairwise incomparable, in the sense of Definition \ref{incomparable naturals}. In this case, if $n\inn$ with $d\leqslant n$ and $(\e_i)_{i=1}^d$ are any signs in $\{-1,1\}$ we call the average
$$ b^* = \frac{1}{n}\sum_{i=1}^d\e_ie_{\ga_i}^*\circ \bar{P}_{E_i}$$
an \ic-average of $\Xi$.

\item[(ii)] The sequence of pairs $(\ga_i, E_i)_{i=1}^d$ is called comparable, if there exist $m\inn$ with $d\leqslant m$, $\{(\eta_1, x_1),\ldots,(\eta_m, x_m)\}\in\mathcal{U}$ and $1\leqslant k_1<\cdots<k_d\leqslant m$ so that the following are satisfied:
    \begin{itemize}

    \item[(a)] $\we(\eta_{k_i}) = \we(\ga_i)$,

    \item[(b)] if $d\geqslant 4$ then $|e_{\ga_i}^*\circ\bar{P}_{E_i}(x_{k_i}) - e_{\ga_j}^*\circ\bar{P}_{E_j}(x_{k_j})| < 1/2^i$ for $2\leqslant i < j\leqslant d-1$.

    \end{itemize}
    In this case, if $n\inn$ with $d\leqslant n$ and $(\e_i)_{i=1}^d$ is a sequence of alternating signs in $\{-1,1\}$ we call the average
    $$ b^* = \frac{1}{n}\sum_{i=1}^d\e_ie_{\ga_i}^*\circ \bar{P}_{E_i}$$
    a \co-average of $\Xi$.

\item[(iii)] The sequence of pairs $(\ga_i, E_i)_{i=1}^d$ is called irrelevant, if there exist $m\inn$ with $d\leqslant m$, $\{(\eta_1, x_1),\ldots,(\eta_m, x_m)\}\in\mathcal{U}$ and $1\leqslant k_1<\cdots<k_d\leqslant m$ so that the following are satisfied:
    \begin{itemize}

    \item[(a)] $\we(\eta_{k_i}) = \we(\ga_i)$ and

    \item[(b)] if $d\geqslant 3$ then $|e_{\ga_i}^*\circ\bar{P}_{E_i}(x_{k_i})| > 16000$ for $2=1,\ldots,d-1$.

    \end{itemize}
    In this case, if $n\inn$ with $d\leqslant n$ and $(\e_i)_{i=1}^d$ are any signs in $\{-1,1\}$ we call the average
$$ b^* = \frac{1}{n}\sum_{i=1}^d\e_ie_{\ga_i}^*\circ \bar{P}_{E_i}$$
    an \ir-average of $\Xi$.

\item[(iv)] Moreover, we call a basic average of $\Xi$, any average of the form
$$b^* = \frac{1}{n}\sum_{i=1}^d\e_i\bar{d}_{\ga_i}^*$$
where $d\leqslant n$, $\ga_i\in\Xi$ with $\ra(\ga_1) < \cdots < \ra(\ga_d)$ and $(\e_i)_{i=1}^d$ are any signs in $\{-1,1\}$. In this case we do not impose any restrictions on the weights of the $\ga_i$'s.  Note that $\bar{d}_{\ga_i}^* = e_{\ga_i}^*\circ \bar{P}_{\{\ra(\ga_i)\}}$, hence basic averages are $\al$-averages.

\end{itemize}
Any average which is of one of the forms defined above, shall be called an \ac-average of $\Xi$.
\end{dfn}

\begin{rmk}\label{subseq preserve type}
A sequence of pairs $(\ga_i, E_i)_{i=1}^d$ can be of none or of more than one of the types described in Definition \ref{def averages}. If it is of any of the first three types, then any of its subsequences is of the same type as well. Moreover, if $b^*$ is an \ac-average of $\Xi$ and $E$ is an interval of $\N$ so that $e_\ga^*\circ \bar{P}_E \neq 0$, then $e_\ga^*\circ \bar{P}_E$ is also an \ac-average of $\Xi$.
\end{rmk}

\begin{prp}\label{building averages}
Let $\Xi$ be a subset of $\bar{\Ga}$, $(\ga_i)_i$ be a sequence in $\Xi$ and $(E_i)_i$ be a sequence of successive intervals of $\N$ so that $\ra(\ga_i)\geqslant\min E_i$ for all $i\inn$ and the set $\{(\we(\ga_i))^{-1}:i\inn\}$ is unbounded. Then there exists an infinite subset $L$ of $\N$ so that for every $d\inn$ and $i_1 <\cdots < i_d$, the sequence $(\ga_{i_j}, E_{i_j})_{j=1}^d$ satisfies either (i), (ii) or (iii) of Definition \ref{def averages}.
\end{prp}

\begin{proof}
Choose $j_i$ so that $\we(\ga_i) = m_{j_i}^{-1}$. A Ramsey argument yields that passing to a subsequence either the $j_i$'s are pairwise incomparable in the sense of Definition \ref{incomparable naturals}, or the sequence $(\sigma^{-1}(i_j))_j$ is a chain of elements of $\mathcal{U}$. In the first case it easily follows that (i) is satisfied. Otherwise, we conclude that there is a sequence of pairs $\{(\eta_k, x_k)\}_{k=1}^\infty$ and a strictly increasing sequence $(d_i)_i$ of $\N$ so that for all $i\inn$, $\sigma^{-1}(j_i) = \{(\eta_k, x_k)\}_{k=1}^{d_i}$. If, passing to a subsequence, for all $i\inn$ $|e_{\ga_i}^*\circ\bar{P}_{E_i}(x_{d_{i}+1})| > 16000$ we conclude that (ii) is satisfied. Otherwise, a compactness arguments yields that passing to a further subsequence (iii) is satisfied.
\end{proof}

\subsection{The space $\X$}

We recursively choose subsets $\De_q$ of $\bar{\De}_q$ as follows: we set $\De_1 = \bar{\De}_1$ and if for $q\inn$ we have chosen the sets $\De_1,\ldots,\De_q$, set $\Ga_q = \cup_{p=1}^q\De_p$ and
\begin{equation*}
\begin{split}
\De_{q+1} =& \left\{(q+1, m_j^{-1}, b^*)\in\bar{\De}_{q+1}: b^*\;\text{is an \ac-average of}\;\Ga_q\right\}\\
&
\begin{split}\cup\left\{\vphantom{2^{\ra(\xi)}}\right.&(q+1, \xi, m_j^{-1}, b^*)\in\bar{\De}_{q+1}: \xi\in\Ga_q,\;b^*\;\text{is an \ac-average of}\;\Ga_q\\
&\left.\text{with size}\; s(b^*) \geqslant \mathcal{N}_{\ra(\xi)}\right\}.
\end{split}
\end{split}
\end{equation*}
For the definition of $\mathcal{N}_{\ra(\xi)}$ see below \eqref{some small but needed change in the definition of the dense subsets}. Note that for all $q$ the set $\De_q$ is non-empty as for $q>1$, $(q,m_1^{-1}, \bar{d}_0^*)\in\De_q$ (recall $\De_1 = \bar{\De}_1 = \{0\}$). We define $\Ga = \cup_{q}\Ga_q$.

\begin{prp}
The set $\Ga$ is a self-determined subset of $\bar{\Ga}$, hence it defines a Bourgain-Delbaen space $\Xbd = \X$ so that the restriction from $\bar{\Ga}$ to $\Ga$ defines a quotient operator $R:\BmT\rightarrow\X$.
\end{prp}

\begin{proof}
We will use Proposition \ref{self-determinacy is all those nice things} (c). As it clearly follows from Proposition \ref{mtbd} and the definition of the set $\Ga$, for every $q\inn$, if $\ga\in\De_{q+1}$ then there is $b^*$ in $\langle\{e_\eta^*\circ \bar{P}_E:\eta\in\Ga_q, E\subset\N\}\rangle$ and $j\inn$ so that either $\bar{c}_\ga^* = (1/m_j)b^*$, or $\bar{c}_\ga^* = e_\eta^* + (1/m_j)b^*$ for some $\eta\in\Ga_q$. We conclude that condition (c) of Proposition \ref{self-determinacy is all those nice things} is indeed satisfied.
\end{proof}

\subsection{Some remarks on the space $\X$}
As we have mentioned earlier, for the space $\X$ we shall use the standard notation $d_\ga$, $c_\ga^*$, $P_E$ etc. Henceforth, whenever we say \ac-average, we shall mean an \ac-average of $\Ga$. Moreover, in the light of Proposition \ref{analysis invariant}, it makes sense to identify any such average either with a basic average of $\Ga$, i.e. $b^* = (1/n)\sum_{i=1}^d\e_id_{\ga_i}^*$, or with a functional $b^* = (1/n)\sum_{i=1}^d\e_ie_{\ga_i}^*\circ P_{E_i}$ so that $d$, $n$, $(E_i)_{i=1}^d$ $(\e_{i})_{i=1}^d$ and $(\ga_{i})_{i=1}^d\in \Ga_q$ satisfy one of (i), (ii) or (iii) of Definition \ref{def averages}. We remark that this is not independent of the set $\bar{\Ga}$. Nevertheless, if $\ga\in\Ga$ with $\ga = (q+1, m_j^{-1}, b^*)$ or $\ga = (q+1, \xi, m_j^{-1}, b^*)$, we can assume that $b^*$ is an \ac-average as it was just described. Furthermore, if $\ga = (q+1, m_j^{-1}, b^*)$ then by \eqref{age one bar}
\begin{subequations}
\begin{equation}
c_\ga^* = \frac{1}{m_j}\frac{1}{n}\sum_{i=1}^d\e_ie_{\ga_i}^*\circ P_{E_i},
\end{equation}
whereas if $\ga = (q+1, \xi, m_j^{-1}, b^*)$ then moreover $\ra(\xi) <\min E_1$, $n > 2^{\ra(\xi)}$ and by \eqref{older age bar}
\begin{equation}
c_\ga^* = e_{\xi}^* + \frac{1}{m_j}\frac{1}{n}\sum_{i=1}^d\e_ie_{\ga_i}^*\circ P_{E_i}.
\end{equation}
\end{subequations}
We also note that $\|i_q\| \leqslant 2$ for all $q$, hence $\|P_E\| \leqslant 4$ for all intervals $E$ of $\N$ which implies that for every \ac-average $b^*$, $\|b^*\| \leqslant 4$.

The following is a restatement of \cite[Proposition 4.5]{AH} in the present setting.

\begin{prp}\label{evaluation analysis prp}
Let $q\inn$ and $\ga\in\De_{q+1}$ with $\we(\ga) = m_j$ and $\ag(\ga) = a \leqslant n_j$. Then there exist natural numbers $0 = p_0 < p_1 < \cdots < p_a = q+1$, elements $\xi_1,\ldots,\xi_a = \ga$ of weight $m_j$ with $\xi_r \in \De_{p_r}$ for $r=1,\ldots,a$ and a very fast growing sequence of \ac-averages $(b_r^*)_{r=1}^a$  with $b_r^*\in B_{p_{r-1},p_r - 1}$ for $r=1,\ldots,a$ such that
\begin{equation}\label{evaluation analysis}
e_\ga^* = \sum_{r=1}^ad^*_{\xi_r} + \frac{1}{m_j}\sum_{r=1}^ab_r^*  = \sum_{r=1}^ad^*_{\xi_r} + \frac{1}{m_j}\sum_{r=1}^ab_r^*\circ P_{(p_{r-1},p_r)}.
\end{equation}
Moreover, if $1\leqslant t < a$, then
\begin{equation}\label{evaluation analysis detail}
e_\ga^* = e_{\xi_t}^* + \sum_{r=t+1}^ad^*_{\xi_r} + \frac{1}{m_j}\sum_{r=t+1}^ab_r^*.
\end{equation}
The form \eqref{evaluation analysis} of $e_\ga^*$ is called the evaluation analysis of $\ga$.
\end{prp}

 A finite inductive argument also yields the following.

\begin{prp}\label{build on vfg}
Let $j\inn$, $1\leqslant a \leqslant n_j$, $0 \leqslant p_0 < p_1 < \cdots < p_a = q+1$ with $j\leqslant p_1$ and $(b_r^*)_{r=1}^a$ be a very fast growing sequence of \ac-averages with $b_r^*\in B_{p_{r-1},p_r - 1}$ for $r=1,\ldots,a$. Then there are $\ga\in\De_{q+1}$ and $\xi_1,\ldots,\xi_a = \ga$, all of weight $m_j$, with $\xi_r \in \De_{p_r}$ for $r=1,\ldots,a$ so that $\ga$ has an evaluation analysis
\begin{equation*}
e_\ga^* = \sum_{r=1}^ad^*_{\xi_r} + \frac{1}{m_j}\sum_{r=1}^ab_r^*.
\end{equation*}
\end{prp}

\subsection{Subspaces and quotients of $\mathfrak{X}_{\mathrm{AH}}$ defined by self-determined subsets}\label{sub self determined of AH}
We mention some results that can be derived by considering subspaces and quotients of the Argyros-Haydon space that are defined by self-determined sets.

\begin{rmk}\label{some further remarks on quotients}
The Argyros-Haydon space $\mathfrak{X}_{\mathrm{AH}}$ from \cite{AH} can also be obtained by finding an appropriate self-determined subset $\Ga^{\mathrm{AH}}$ of $\bar{\Ga}$, hence the space $\mathfrak{X}_{\mathrm{AH}}$ is a quotient of $\BmT$ as well.
\end{rmk}

\begin{rmk}\label{remark concerning scalar plus compact on complements of self-determined}
In \cite{KL} T. Kania and J. N. Laustsen choose a $\mathscr{L}_\infty$-subspace $Y$ of $\mathfrak{X}_{\mathrm{AH}}$ so that every bounded linear operator $T:Y\rightarrow\mathfrak{X}_{\mathrm{AH}}$ is a scalar multiple of the inclusion plus a compact operator. Following their notation, the Bourgain-Delbaen space $\mathfrak{X}_{\mathrm{AH}}$ is defined by a set $\Ga^{\mathrm{AH}}$ and the space $Y$ is the closed linear span of a subsequence $(d_{\ga})_{\ga\in\Ga'}$ of the basis, where $\Ga'$ is an appropriately chosen subset of $\Ga^{\mathrm{AH}}$. They prove that this set $\Ga'$ has the property that whenever $\ga\in\Ga'$, then $d_\ga|_{\Ga^{\mathrm{AH}}\setminus\Ga'} = 0$ (\cite[Lemma 2.5]{KL}), which by Proposition \ref{self-determinacy is all those nice things} (d) is equivalent to $\Ga^{\mathrm{AH}}\setminus\Ga'$ being self-determined. Actually, this is the only property of $\Ga'$ that they use to prove the properties of the space $Y$. Hence, they have proved the result below.
\end{rmk}

\begin{prp}\label{scalar plus compact on complements proposition}
Let $\Xbd = \mathfrak{X}_{\mathrm{AH}}$ be the $\mathscr{L}_\infty$-space with the scalar-plus-compact property from \cite{AH}. Let also $\Ga'$ be a self-determined subset of $\Ga^{\mathrm{AH}} = \cup_q\Ga_q$ and $Y = \overline{\langle d_\ga:\ga\in\Ga^{\mathrm{AH}}\setminus \Ga'\rangle}$. Then, every bounded linear operator $T:Y\rightarrow \mathfrak{X}_{\mathrm{AH}}$ is a multiple of the inclusion plus a compact operator.
\end{prp}

\begin{lem}\label{lemma about inclusion plus compact subsets}
Let $X$ be a Banach space with a basis $(e_i)_i$ and  assume that $A$ is a subset of $\N$ so that every bounded linear operator $T:Y = \overline{\langle\{ e_i:i\in A\}\rangle}\rightarrow X$ is a multiple of the inclusion plus a compact operator. If $B$ is a subset of $\N$ so that $Y$ isomorphically embeds into $Z  = \overline{\langle\{ e_i:i\in B\}\rangle}$, then the set $A\setminus B$ is finite. In particular, if the set $A\setminus B$ is infinite, then every bounded linear operator $T:Y\rightarrow Z$ is compact.
\end{lem}

\begin{proof}
We may clearly assume that the basis is seminormalized. If the set $A\setminus B$ is infinite, it contains an infinite sequence $(n_k)_k$. Let $T:Y\rightarrow Z$ be a bounded linear operator, then $T = \la I_{Y,X} + K$ with $K$ a compact operator. As $n_k\notin B$, we obtain $e_{n_k}^*(Te_{n_k}) = 0$ for all $k\inn$. By the compactness of $K$ and passing to a subsequence, there is $x_0$ in $X$ so that $(Te_{n_k} - \la e_{n_k})_k$ converges to $x_0$ in norm, which yields  $ 0 = \lim_ke_{n_k}^*(Te_{n_k}) = \la + \lim_ke_{n_k}^*(x_0) = \la$, therefore $T = K$.
\end{proof}

Proposition \ref{scalar plus compact on complements proposition} and Lemma \ref{lemma about inclusion plus compact subsets} immediately yield the following.

\begin{cor}\label{from bigger to smaller compact}
Let $\Ga_1$, $\Ga_2$ be two self-determined subsets of $\Ga^{\mathrm{AH}}$ so that $\Ga_2\setminus \Ga_1$ is infinite. If $Y = \overline{\langle d_\ga:\ga\in\Ga^{\mathrm{AH}}\setminus \Ga_1\rangle}$ and $Z = \overline{\langle d_\ga:\ga\in\Ga^{\mathrm{AH}}\setminus \Ga_2\rangle}$, then every bounded linear operator $T: Y\rightarrow Z$ is compact.
\end{cor}

It is not very difficult to find a continuum of self-determined subsets of $\Ga^{\mathrm{AH}}$ that pairwise satisfy the assumptions of Corollary \ref{from bigger to smaller compact}. We choose these sets in such a way that the corresponding quotients have similar properties as well. Recall that the set $\Ga^{\mathrm{AH}}$ is built using a sequence of weights $(m_j,n_j)_j$, where the even weights are used freely to define new coordinates, whereas some restrictions are applied to the odd weights. For each infinite subset $L$ of $\N$, one can define a self-determined subset $\Ga_L$ of $\Ga^{\mathrm{AH}}$, only using the weights $(m_{2j},n_{2j})_{j\in L}$. This is done so that the weights $(m_{4j},n_{4j})_{j\in L}$ are used unconditionally, whereas the weights $(m_{4j-2},n_{4j-2})_{j\in L}$ are used conditionally, i.e. they assume the role of the odd weights in the construction from \cite{AH}. For this last part, a coding function specific to the subset $\Ga_L$ needs to be used.

We observe that the subset $\Ga_L$ of $\Ga^{\mathrm{AH}}$ induces a Bourgain-Delbaen space which is qualitatively identical to the space $\mathfrak{X}_{\mathrm{AH}}[(\mathscr{A}_{n_j},1/m_j)_{j\in2L}]$ defined in \cite[Subsection 10.2]{AH}. Hence, if we choose a continuum $\{L_\alpha:\alpha\in\mathfrak{c}\}$ of infinite subsets of $\N$, with pairwise finite intersections, set $\Ga_\alpha = \Ga_{L_\alpha}$, and define the spaces $Y_\alpha = \overline{\langle\{d_\ga:\ga\in\Ga\setminus \Ga_\alpha\}\rangle}$, then the spaces $\{Y_\alpha: \alpha\in\mathfrak{c}\}$ satisfy the assumptions of Corollary \ref{from bigger to smaller compact} and the spaces $X_\alpha = \mathfrak{X}_{\mathrm{AH}}/Y_\alpha$, $\alpha\in\mathfrak{c}$ satisfy the conclusion of \cite[Theorem 10.4]{AH}, i.e. the following holds.

\begin{thm}\label{spc quotients of AH}
There is a continuum of $\mathscr{L}_\infty$-subspaces  $\{Y_\alpha:\alpha \in\mathfrak{c}\}$ of $\mathfrak{X}_{\mathrm{AH}}$, satisfying the following.
\begin{itemize}

\item[(i)] Each space $Y_\alpha$ has the scalar-plus-compact property and for every $\alpha \neq \beta$ every bounded linear operator $T:Y_\alpha\rightarrow Y_\beta$ is compact.

\item[(ii)] Each $X_\alpha = \mathfrak{X}_{\mathrm{AH}}/Y_\alpha$ is a hereditarily indecomposable space with the scalar-plus-compact property and for every $\alpha\neq\beta$ every bounded linear operator $T:X_\alpha\rightarrow X_\beta$ is compact.

\end{itemize}
\end{thm}

Observe that for fixed $\alpha$, all three spaces $\mathfrak{X}_{\mathrm{AH}}$, $Y_\alpha$ and $X_\alpha = \mathfrak{X}_{\mathrm{AH}}/Y_\alpha$ are hereditarily indecomposable $\mathscr{L}_\infty$-spaces with the scalar-plus-compact property.

\begin{rmk}\label{non-reflexive as quotient of AH}
A version $\tilde{\mathfrak{X}}_{\mathfrak{nr}}$ of the space $\X$ can be obtained as a quotient of a version $\tilde{\mathfrak{X}}_{\mathrm{AH}}$ of the space $\mathfrak{X}_{\mathrm{AH}}$ (the difference being similar to the one stated in Remark \ref{we need more convex combinations, MORE}). This is achieved by defining a self-determined subset $\tilde \Ga$ of $\Ga^{\mathrm{AH}}$ defined only on coordinates with even weight. This construction also satisfies that if $Y$ is the kernel of the quotient operator $R:\tilde{\mathfrak{X}}_{\mathrm{AH}}\rightarrow\tilde{\mathfrak{X}}_{\mathfrak{nr}}$, then $\tilde{\mathfrak{X}}_{\mathrm{AH}}$, $\tilde{\mathfrak{X}}_{\mathfrak{nr}}$, and $Y$ all have the scalar-plus-compact property.
\end{rmk}

\begin{rmk}
A self-determined subset $\Ga$ of $\Ga^{\mathrm{AH}}$ can be chosen so that the corresponding quotient is isomorphic to $c_0$. This set can be chosen starting with a random point $\ga$ and then only allow operations that result in new coordinates with age at most one. One can then choose a self-determined subset $\Ga'$ of $\Ga^{\mathrm{AH}}$, almost disjoint to $\Ga$, with the same properties. We set $Y = \overline{\langle\{d_\ga:\ga\notin \Ga\}\rangle}$ and $Y' = \overline{\langle\{d_\ga:\ga\notin \Ga'\}\rangle}$. Then $Y$, $Y'$ are subspaces of $\mathfrak{X}_{\mathrm{AH}}$ with the scalar plus compact property so that every operator from one to the other is compact (by Corollary \ref{from bigger to smaller compact}). However, both spaces $\mathfrak{X}_{\mathrm{AH}}/Y$ and $\mathfrak{X}_{\mathrm{AH}}/Y'$ are isomorphic to $c_0$.
\end{rmk}

\section{The $\al$-index}
A tool that has been used in recent constructions involving saturation under constraints is the $\al$-index of a block sequence (\cite{ABM}, \cite{AM1}, \cite{AM2} and more).  This index helps characterize what spreading models a given block sequence admits. However, due to the Bourgain-Delbaen construction and the mixed-Tsirelson setting, in the space $\X$ the index does not fully determine spreading models. Nevertheless, it remains an integral part of the study of spaces constructed with the method of saturation under constraints.

\begin{dfn}\label{idx def}
Let $(x_k)_k$ be a block sequence in $\X$, so that for every very fast growing sequence of \ac-averages $(b_j^*)_j$ and every subsequence $(x_{k_j})_j$ of $(x_k)_k$,
$$\lim_j|b_j^*(x_{k_j})| = 0.$$
Then we say that the $\al$-index of $(x_k)_k$ is zero and write $\adxk = 0$. Otherwise, we write $\adxk > 0$.
\end{dfn}

\begin{prp}\label{adx char}
Let $(x_k)_k$ be a block sequence in $\X$. The following assertions are equivalent.
\begin{itemize}

\item[(i)] The $\al$-index of $(x_k)_k$ is zero.

\item[(ii)] For every $\e>0$ there exist $k_0$ and $j_0\inn$ so that for every $k\geqslant k_0$, interval $E$ of $\N$ and \ac-average $b^*$ with $s(b^*)\geqslant j_0$, $|b^*(P_Ex_k)| < \e$.

\end{itemize}
\end{prp}

\begin{rmk}\label{if adx zero small on bounded sums of vfg}
Using the above characterization and that for every \ac-average $b^*$, $\|b^*\| \leqslant 4$, it easily follows that if $(x_k)_k$ is a block sequence in $\X$ with $\adxk = 0$, then for every $a\inn$ and $\e > 0$ there exists $k_0\inn$ so that for all $k\geqslant k_0$ and very fast growing sequence of \ac-averages $(b_r^*)_{r=1}^a$, $\sum_{r=1}^a|b_r^*(x_k)| < 4\|x_k\| + \e$
\end{rmk}

The proof of the next result is an easy consequence of Definition \ref{idx def} and Proposition \ref{build on vfg}.

\begin{prp}\label{adx positive}
Let $(x_k)_k$ be a seminormalized block sequence in $\X$ with $\adxk > 0$. Then there exist $\theta > 0$ and a subsequence of $(x_k)_k$, again denoted by $(x_k)_k$, so that for all natural numbers $j\leqslant k_1 < \cdots < k_{n_j}$ and scalars $(\la_i)_{i=1}^{n_j}$,
$$\left\|\sum_{i=1}^{n_j}\la_ix_{k_i}\right\| \geqslant \theta\frac{1}{m_j}\sum_{i=1}^{n_j}|\la_i|.$$
\end{prp}

\begin{prp}\label{adx zero}
Let $(x_k)_k$ be a normalized block sequence in $\X$ with $\adxk = 0$ and $\lim_k\sup_{\ga\in\Ga}|d_\ga^*(x_k)| = 0$. Then $(x_k)_k$ has a subsequence, which we also denote by $(x_k)_k$, that generates a spreading model isometric to $c_0$. Moreover, there exists a strictly increasing sequence of natural numbers $(j_k)_k$ so that for every natural numbers $n \leqslant k_1 <\cdots<k_n$, scalars $(\la_i)_{i=1}^n$, $\ga\in\Ga$ with $\we(\ga) = m_j^{-1} > m_{j_n}^{-1}$ and interval $E$ of $\N$,
\begin{equation}\label{sums ris}
\left|e_\ga^*\circ P_E\left(\sum_{i=1}^n\la_ix_{k_i}\right)\right| \leqslant \frac{C}{m_j}\max_{1\leqslant i\leqslant n}|\la_i|,
\end{equation}
where $C = 8$.
\end{prp}

\begin{proof}
Using Proposition \ref{adx char} and $\lim_k\sup_{\ga\in\Ga}|d_\ga^*(x_k)| = 0$, we pass to a subsequence of $(x_k)_k$, again denoted by $(x_k)_k$, and choose a strictly increasing sequence of natural numbers $(j_k)$ so that the following are satisfied:
\begin{itemize}

\item[(i)] for every $k\inn$, $j_{k+1} > \max\supp x_k$,

\item[(ii)] for every $k\inn$, $\sum_{m\geqslant k}\sup_{\ga\in\Ga}|d_\ga^*(x_m)| < 1/(2km_{j_{k}}n_{j_k})$ and

\item[(iii)] for every $k_0$, $k\inn$ with $k\geqslant k_0$, every interval $E$ of $\N$ and every \ac-average $b^*$ with $s(b^*) \geqslant \min\supp x_k$, $|b^*(P_Ex_k)| < 1/(2k_0n_{j_{k_0}})$.

\end{itemize}

We claim that $(x_k)_k$ satisfies the conclusion. By induction on $q$ we shall prove the following: for every $\ga\in\Ga_q$, interval $E$ of $\N$, natural numbers $n \leqslant k_1 <\cdots < k_n$ and scalars $\la_1,\ldots,\la_n$ in $[-1,1]$:
\begin{equation}\label{green mushroom}
\left|e_\ga^*\left(\sum_{i=1}^n\la_ix_{k_i}\right)\right| < 1 + \frac{15}{m_{j_n}}\quad\text{and}\quad\left|e_\ga^*\circ P_E\left(\sum_{i=1}^n\la_ix_{k_i}\right)\right| \leqslant 7.
\end{equation}
If moreover $\we(\ga)  = m_j^{-1}$ with $j < j_n$, then
\begin{equation}\label{one up}
\left|e_\ga^*\circ P_E\left(\sum_{i=1}^n\la_ix_{k_i}\right)\right| <  \frac{8}{m_j}.
\end{equation}
The desired conclusion clearly follows from the above.

The case $q=1$ is an easy consequence of the definition of $\De_1$. Assume now that $q$ is such that the conclusion holds for every $\ga\in\Ga_q$ and let $\ga\in \Ga_{q+1}$ with $\we(\ga) = m_j^{-1}$ and $E$ be an interval of $\N$. Let $e_\ga^* = \sum_{t=1}^ad_{\xi_t}^* + (1/m_j)\sum_{t=1}^ab_t^*$ be the evaluation analysis of $\ga$, according to Proposition \ref{evaluation analysis}. Then $1\leqslant a \leqslant n_j$ and $(b_t^*)_{t=1}^a$ is a very fast growing sequence of \ac-averages of $\Ga_q$. Assuming that $\ra(\ga)\geqslant\min\supp x_{k_1}$ (otherwise the estimates appearing in \eqref{green mushroom} and \eqref{one up} are all zero), set $t_0 = \min\{t:\;\max\supp b_q^*\geqslant \min\supp x_{k_1}\}$. The inductive assumption easily implies that for $1\leqslant i_0\leqslant n$,
\begin{equation}\label{grumble grumble}
\left|b^*_{t_0}\circ P_E\left(\sum_{i=1}^{i_0}\la_ix_{k_i}\right)\right| \leqslant 7.
\end{equation}
We shall distinguish three cases concerning the weight of $\ga$.

{\em Case 1:} $j<j_{k_1}$. Since the sequence $(b^*_q)_{q=1}^d$ is very fast growing, for $t>t_0$ we have $s(b_t^*) > \max\supp b_{t_0}^* \geqslant \min\supp x_{k_1}$. Also $a\leqslant n_j < n_{j_{k_1}}$ and $n \leqslant k_1$, hence by (iii) we conclude:
\begin{equation}\label{goomba}
\sum_{t=t_0 + 1}^a\left|b_t^*\circ P_E\left(\sum_{i=1}^n\la_ix_{k_i}\right)\right| < a\frac{n}{2k_1n_{j_{k_1}}} \leqslant \frac{1}{2}.
\end{equation}
By (ii) we obtain
\begin{equation}\label{parakoopa}
\left|\sum_{t = 1}^ad_{\xi_{t}}^*\circ P_E \left(\sum_{i=1}^n\la_ix_{k_i}\right)\right| < a\frac{n}{2k_1m_{j_{k_1}}n_{j_{k_1}}} \leqslant \frac{1}{2m_{j_{k_1}}}.
\end{equation}
Combining \eqref{grumble grumble} with \eqref{goomba} and \eqref{parakoopa}:
\begin{equation}
\left|e_\ga^*\circ P_E\left(\sum_{i=1}^n\la_ix_{k_i}\right)\right| <  \frac{8}{m_j}.
\end{equation}
This concludes the proof of the first case and also \eqref{one up} of the inductive assumption (the first part of \eqref{green mushroom} follows if we set $E= \N$ and use $m_1\geqslant 8$).

{\em Case 2:} there is $1\leqslant i_0 < n$ so that $j_{k_{i_0}} \leqslant j < j_{k_{i_0+1}}$. Arguing in a similar manner as in the previous case, we obtain
\begin{equation}\label{koopa}
\left|e_\ga^*\circ P_E\left(\sum_{i>i_0}\la_ix_{k_i}\right)\right| < \frac{8}{m_{j_{k_{i_0}}}}.
\end{equation}
Note that (i) and $\we(d_{\xi_t}) = m_j^{-1}$ imply $\ra(\xi_t) > \max\supp x_{i_0-1}$ which yields $\sum_{t = 1}^ad_{\xi_{t}}^*\circ P_E (\sum_{i<i_0}\la_ix_{k_i}) = 0$ and $\sum_{t=t_0 + 1}^ab_t^*\circ P_E(\sum_{i<i_0}\la_ix_{k_i}) = 0$. Combining this with \eqref{grumble grumble}:
\begin{equation}\label{orange}
\left|e_\ga^*\circ P_E\left(\sum_{i<i_0}\la_ix_{k_i}\right)\right| = \frac{1}{m_j}\left|b^*_{t_0}\circ P_E\left(\sum_{i<i_0}\la_ix_{k_i}\right)\right|\ <  \frac{7}{m_{j_{k_{i_0}}}}.
\end{equation}
As $\|e_\ga^*\circ P_E\| \leqslant 4$ we obtain that $|e_\ga^*\circ P_E(x_{k_{i_0}})| \leqslant 4$, which in conjunction with \eqref{koopa} and \eqref{orange} yields,
\begin{equation}
\left|e_\ga^*\circ P_E\left(\sum_{i=1}^n\la_ix_{k_i}\right)\right| \leqslant 4+ \frac{15}{m_{j_{k_{i_0}}}} \leqslant 4 + \frac{15}{m_1} \leqslant 7.
\end{equation}
Similarly, for $E = \N$ and using $|e_\ga^*(x_{k_{i_0}})| \leqslant 1$ we obtain,
\begin{equation}
\left|e_\ga^*\left(\sum_{i=1}^n\la_ix_{k_i}\right)\right| < 1 + \frac{15}{m_{j_{k_{i_0}}}} \leqslant 1 + \frac{1}{m_{j_n}}.
\end{equation}
This concludes the proof of the second case. The third case, in which $j\geqslant j_{k_n}$, is treated in a similar manner as the second one.
\end{proof}

\begin{rmk}
\label{comparing spreading models with dual methods}
We point out that the space without reflexive subspaces constructed in \cite{AM2} admits precisely three spreading models in every subspace, namely the unit vector basis of $\ell_1$, the unit vector basis of $c_0$, and the summing basis of $c_0$. This is no longer true for the space $\X$ presented in this paper, as this space admits a large variety of spreading models. This is due to the $\mathscr{L}_\infty$ structure and mainly due to the mixed-Tsirelson frame used to define the norm, as opposed to the Tsirelson frame used in \cite{AM2}. We also point out that in \cite{AM2} the $\al$-index alone is sufficient to fully describe the spreading models admitted by a block sequence. Here, this is no longer the case and the condition $\adxk = 0$ is not sufficient for a sequence to generate a $c_0$ spreading model and $\lim_k\sup_{\ga\in\Ga}|d_\ga^*(x_k)| = 0$ is necessary as well. As it was shown in the proof \cite[Proposition 10.1]{AH}, the sequence $(y_q)_q$, with $y_q = \sum_{\ga\in\De_q}d_\ga$, generates an $\ell_1$ spreading model. The same sequence in $\X$ generates an $\ell_1$ spreading model as well, however it can be shown that $\adyq = 0$. In the special case when $(x_k)_k$ is a subsequence of the basis $(d_\ga)_{\ga\in\Ga}$, we have $\adxk = 0$ and $\lim_k\sup_{\ga\in\Ga}|d_\ga^*(x_k)|\neq 0$. However, $(x_k)_k$ has a subsequence generating a $c_0$ spreading model. This is proved by replacing the condition $\lim_k\sup_{\ga\in\Ga}|d_\ga^*(x_k)| = 0$ with the conclusion of the following lemma.
\end{rmk}

\begin{lem}\label{ramsey on basis}
Let $\{\ga_k:k\inn\}$ be an infinite subset of $\Ga$. Then there exists an infinite subset $L$ of $\N$ satisfying the following: for every $\ga\in\Ga$, if $e_\ga^* = \sum_{r=1}^ad_{\xi_r}^* + (1/m_j)\sum_{r=1}^ab_r^*$ is the evaluation analysis of $\ga$, then the set $\{\xi_r:r=1,\ldots,a\}\cap\{\ga_k:k\in L\}$ is at most a singleton.
\end{lem}

\begin{proof}
For $\eta$, $\ga\in\Ga$ with $\ra(\eta) < \ra(\ga)$, we shall say that $\eta$ is in the analysis of $\ga$, if $(d^*_{\xi_r})_{r=1}^a$ is the sequence appearing in the evaluation analysis of $\ga$ as in \eqref{evaluation analysis}, then there is some $1\leqslant r\leqslant a$ so that $\xi_r = \eta$. Note that \eqref{evaluation analysis detail} implies that this property is transitive and it also easily follows that there exists no infinite chain with this property.

By passing to an infinite subset we may assume that $(\ra(\ga_k))_k$ is strictly increasing and using an easy Ramsey argument we may also assume that for $k < m$, $\ga_k$ is not in the analysis of $\ga_m$, which implies the desired result.
\end{proof}

An application of the above lemma and arguments similar to those used in the proof of Proposition \ref{adx zero} yield the next result.

\begin{prp}\label{basis c0}
Let $(d_{\ga_k})_k$ be a subsequence of the basis $(d_\ga)_{\ga\in\Ga}$ of $\X$. Then it admits a subsequence generating an isometric $c_0$ spreading model. Furthermore, there exists a constant $C>0$ so that for every natural numbers $n \leqslant k_1 <\cdots<k_n$, scalars $(\la_i)_{i=1}^n$, $\ga\in\Ga$ with $\we(\ga) = m_j^{-1} > m_{j_n}^{-1}$ and interval $E$ of $\N$,
\begin{equation}\label{sums basis ris}
\left|e_\ga^*\circ P_E\left(\sum_{i=1}^n\la_id_{\ga_{k_i}}\right)\right| \leqslant \frac{C}{m_j}\max_{1\leqslant i\leqslant n}|\la_i|.
\end{equation}
If the set $\{(\we({\ga_k}))^{-1}:k\inn\}$ is unbounded, then $C=8$. Otherwise there is $j_0\inn$ with $ C = 2 + m_{j_0}$.
\end{prp}

\begin{lem}\label{tools to build on c0}
Let $(x_k)_k$ be a block sequence in $\X$ generating a $c_0$ spreading model and let $(\ga_k)_k$ be a sequence in $\Ga$ so that $|e_{\ga_k}^*(x_k)| > (3/4)\|x_k\|$ for all $k\inn$. If the set $\{(\we(\ga_k))^{-1}:k\inn\}$ is bounded, then there exist $\e>0$, an infinite subset $L$ of $\N$ and a sequence $(\eta_k)_{k\in L}$ of $\Ga$, so that $|d_{\eta_k}^*(x_k)| > \e$ for all $k\in L$.
\end{lem}

\begin{proof}
Passing to a subsequence, there are $j\inn$ and $1\leqslant a\leqslant n_j$ so that each $\ga_k$ has an evaluation analysis $e_{\ga_k}^* = \sum_{r=1}^ad^*_{\xi_r^k} + (1/m_j)\sum_{r=1}^ab_{k,r}^*$. Since $(x_k)_k$ generates a $c_0$ spreading model, by Proposition \ref{adx positive} we conclude that $\adxk = 0$. By Remark \ref{if adx zero small on bounded sums of vfg} we can assume that $|(1/m_j)\sum_{r=1}^ab_{k,r}^*(x_k)| < (5/m_j)\|x_k\| \leqslant (5/8)\|x_k\|$, which yields $|\sum_{r=1}^ad^*_{\xi_r^k}(x_k)| > 1/8\|x_k\|$, for all $k\inn$. Setting $\e = \inf_k\|x_k\|/(8a)$, the result easily follows.
\end{proof}

\begin{lem}\label{first step building}
Let $(x_k)_k$ be a block sequence in $\X$ generating a $c_0$ spreading model. Then there are $\e > 0$ and a subsequence of $(x_k)_k$, again denoted by $(x_k)_k$, so that for every natural numbers $n\leqslant k_1 < \cdots < k_n$ and sequence of alternating signs $(\e_i)_{i=1}^n$, if $y = \sum_{i=1}^n\e_ix_i$ there is an \ac-average $b^*$ of size $s(b^*) = n$ so that $b^*(y) > \e$.
\end{lem}

\begin{proof}
Choose a sequence $(\ga_k)_k$ in $\Ga$ so that $|e_{\ga_k}^*(x_k)| > (3/4)\|x_k\|$ for all $k\inn$. Passing to a subsequence, and perhaps considering the sequence $(-x_k)_k$, we may assume that $e_{\ga_k}^*(x_k) > (3/4)\|x_k\|$ for all $k\inn$. If the set $\{(\we(\ga_k))^{-1}:k\inn\}$ is unbounded, set $E_k = \ran x_k$ and pass to a subsequence satisfying the conclusion of Proposition \ref{building averages}. Setting $\e = \inf_k\|x_k\|/4$, it easily follows that for $n \leqslant k_1 < \cdots < k_n$ and alternating signs $(\e_i)_{i=1}^n$, $b^* = (1/n)\sum_{i=1}^n\e_ie_{\ga_{k_i}}^*\circ P_{\ran x_{k_i}}$ is the desired \ac-average. Otherwise, we apply Lemma \ref{tools to build on c0} and argue in a similar manner.
\end{proof}

\begin{rmk}\label{first step building better}
The proof of Lemma \ref{first step building} actually yields that in the case $(x_k)_k$ generates a $c_0$ spreading model and $\lim_k\sup_{\ga\in\Ga}|d_\ga^*(x_k)| = 0$, if $(\ga_k)_k$ satisfies $e_{\ga_k}^*(x_k) > (3/4)\|x_k\|$ for all $k$, we can choose a subsequence of $(x_k)_k$, again denoted by $(x_k)_k$, so that for any natural numbers $n\leqslant k_1 < \cdots < k_n$ and sequence of alternating signs $(\e_i)_{i=1}^n$, if $y = \sum_{i=1}^n(\e_i/e_{\ga_i}^*(x_i))x_i$ then  $b^* = (1/n)\sum_{i=1}^n\e_ie_{\ga_i}^*P_{\ran x_{k_i}}$ an \ac-average of size $s(b^*) = n$ so that $b^*(y) = 1$.
\end{rmk}

It immediately follows that if $(x_k)_k$ is a sequence generating a $c_0$ spreading model, then it has a further block sequence $(y_k)_k$ with $\adyk > 0$, hence by Proposition \ref{adx positive} we deduce the following.

\begin{cor}\label{no c0}
The space $\X$ does not contain $c_0$.
\end{cor}

\section{Exact pairs and dependent sequences}
In this section we define exact pairs and dependent sequences and we also show that they can be found in every block subspace. They are important tools used in the sequel to deduce all the properties of the space. The definition of a dependent sequence is based on that from \cite{AH} and has been slightly modified in order to obtain a stronger result.

\subsection{Rapidly increasing sequences and $\ell_1^n$-averages} We recall the definition of a rapidly increasing sequence (RIS), state the basic inequality, for which we do not include a proof (for details see \cite[Section 5]{AH}), and also remind the notion of normalized $\ell_1^n$-averages (see also \cite[Section 8]{AH}). The auxiliary space used for the basic inequality is $T[(\mathscr{A}_{3n_j},m_j^{-1})_j]$ (see \cite[Sections 2.4]{AH}).

\begin{dfn}\label{ris def}
A (finite or infinite) block sequence $(x_k)_k$ is called a $C$-rapidly increasing sequence, or a $C$-RIS, where $C\geqslant 1$, if there is a strictly increasing sequence of natural numbers $(j_k)_k$ , so that the following hold:
\begin{itemize}

\item[(i)] $\|x_k\| \leqslant C$,

\item[(ii)] $j_{k+1} > \max\supp x_k$ and

\item[(iii)] $|e_\ga^*(x_k)| < C/m_j$ whenever $\we(\ga) = m_j^{-1}$ and $j < j_k$,  for all $k$.

\end{itemize}
\end{dfn}

\begin{rmk}\label{on subseq ris}
Note that if an infinite block sequence satisfies (i) and (iii) of Definition \ref{ris def}, for some $C$ and a strictly increasing sequence $(j_k)_k$, then it has a subsequence which is a $C$-RIS.
\end{rmk}

The following Proposition has been essentially proven in \cite{AH}. It is a consequence of the basic inequality and it follows by combining \cite[Corollary 5.5]{AH} with \cite[Proposition 5.6]{AH}. Statement \eqref{raspberries} in particular follows from applying \cite[Lemma 5.3]{AH} to \cite[Proposition 5.6 (1)]{AH}.
\begin{prp}\label{all basics in one}
If $(x_k)_k$ is a $C$-RIS, then for any scalars $(\la_k)_k$ we have
\begin{equation}\label{ris dominated}
\left\|\sum_k\la_kx_k\right\| \leqslant 10C\left\|\sum_k\la_ke_k\right\|_{T[(\mathscr{A}_{3n_j},m_j^{-1})_j]},
\end{equation}
where the right-hand norm is taken in $T[(\mathscr{A}_{3n_j},m_j^{-1})_j]$. More precisely, if $j\inn$ and $(x_k)_{k=1}^{n_j}$ is a $C$-RIS, then:
\begin{equation}\label{currant}
\left\|\frac{m_j}{n_j}\sum_{k=1}^{n_j}x_k\right\| \leqslant 10C
\end{equation}
and if $\ga\in\Ga$ with $\we(\ga) = m_i^{-1}$ and $E$ is an interval of $\N$, then:
\begin{equation}\label{raspberries}
\left|e_\ga^*\circ P_E\left(\frac{m_j}{n_j}\sum_{k=1}^{n_j}x_k\right)\right| \leqslant \left\{
\begin{array}{ll}
\frac{112C}{m_i},& \text{if } i < j\;\text{and}\\
\\
\frac{16Cm_j}{n_j} + \frac{24Cm_j}{m_i} + \frac{80C}{m_i},& \text{if}\;i \geqslant j.
\end{array}
\right.
\end{equation}
\end{prp}

\begin{dfn}
An element $x$ of $\X$ will be called a $C$-$\ell_1^n$-average if there exists a block sequence $(x_k)_{k=1}^n$ in $\X$ such that $x=(1/n)\sum_{k=1}^nx_k$ and $x_k \leqslant C$ for all $k$. We say that $x$ is a normalized $C$-$\ell_1^n$ average if, in addition, $\|x\| =1$.
\end{dfn}

Proposition \ref{adx positive} implies that a sequence $(x_k)_k$ with positive $\al$-index supports normalized $C$-$\ell_1^n$-averages. This can be deduced using an argument similar to that in the proof of \cite[Lemma II.22, page 33]{ATo}.

\begin{lem}\label{averages if adx posititive}
Let $(x_k)_k$ be a seminormalized block sequence in $\X$ with $\adxk > 0$. Then for every $C>1$ and $n\inn$ there exist further normalized block vectors $(y_k)_{k=1}^n$ of $(x_k)_k$, so that $y = (1/n)\|\sum_{k=1}^ny_k\| \geqslant 1/C$. In particular, the vector $(1/\|y\|)y$ is a normalized $C$-$\ell_1^n$-average.
\end{lem}

A standard argument yields the following result (for a proof see e.g. \cite[Lemma 3.3]{ABM}).

\begin{lem}\label{on averages adx zero}
Let $y$ be a normalized $C$-$\ell_1^n$-average and $b^*$ be an \ac-average. Then $|b^*(y)| < 4C/s(b^*) + 8C/n$. In particular, if $(y_k)_k$ is a block sequence in $\X$ so that each $y_k$ is a $C$-$\ell_1^{r_k}$-average with $(r_k)_k$ strictly increasing, then $\adyk = 0$ and $\lim_k\sup_{\ga\in\Ga}|d_\ga^*(y_k)| = 0$.
\end{lem}

\begin{rmk}\label{geburtstagstorte}
As it is shown in \cite[Lemma 8.4]{AH}, a sequence of $C$-$\ell_1^n$-averages with strictly increasing $n$'s, has a subsequence which is a $2C$-RIS.
\end{rmk}

A standard argument using either Lemma \ref{on averages adx zero} and Proposition \ref{adx zero} or Remark \ref{geburtstagstorte} and \eqref{currant}, yields the following.
\begin{prp}\label{shrinking}
The FDD of $\X$ is shrinking. In particular, $\Xstar$ is isomorphic to $\ell_1$.
\end{prp}

\begin{prp}\label{c0 sm in every subspace}
Every block subspace $X$ of $\X$ contains a normalized block sequence $(y_k)_k$ with $\adyk = 0$ and $\lim_k\sup_{\ga\in\Ga}|d_\ga^*(y_k)| = 0$. In particular, every subspace of $\X$ admits a $c_0$ spreading model.
\end{prp}

\begin{proof}
As the sequence $(d_\ga^*)_{\ga\in\Ga}$ is weak-star null, Corollary \ref{no c0} implies that there is a further block subspace $Z$ of $X$ so that for every bounded block sequence $(z_k)_k$ in $Z$, $\lim_k\sup_{\ga\in\Ga}|d_\ga^*(z_k)| = 0$. We fix any normalized block sequence $(z_k)_k$ in $Z$. If $\adzk = 0$ then this is the desired sequence. Otherwise, $\adzk > 0$ and Lemmas \ref{averages if adx posititive} and \ref{on averages adx zero} yield that there is a further block sequence $(y_k)_k$ of $(z_k)_k$ satisfying the conclusion. The second part follows from Proposition \ref{adx zero}.
\end{proof}

\begin{prp}\label{not L infty saturated}
The space $\X$ is not $\mathscr{L}_\infty$-saturated. More precisely, if $X$ is generated by a skipped-block sequence of $\X$, then $X$ does not contain a $\mathscr{L}_\infty$-subspace.
\end{prp}

\begin{proof}
Assume that there is $X$ as in the statement that contains a $\mathscr{L}_\infty$-subspace $Y$. Proposition \ref{shrinking} and \cite[Corollary Page 182]{LS} yield that $Y^*$ is isomorphic to $\ell_1$. By Proposition \ref{c0 sm in every subspace} we may then find a normalized block sequence in $Y$, which is a perturbation of a skipped block sequence $(x_i)_i$, and generates a $c_0$ spreading model. By Lemma \ref{first step building} there are a normalized block sequence $(y_i)_i$ of $(x_i)_i$, $\e > 0$, and a very fast growing sequence of \ac-averages $(b_i^*)_i$ so that $b_i^*(y_i) > \e$ for all $i\inn$. It follows that if $f_i$ is the restriction of $b_i^*$ onto $Y$, then $(f_i)_i$ has a subsequence equivalent to the unit vector basis of $\ell_1$. By Proposition \ref{build on vfg}, for every $j\inn$ we can find $i_1 < \cdots < i_{n_j}$ and $(d_{\xi_r}^*)_r$ in the annihilator of $X$, and hence also of $Y$, so that $\|\sum_{r=1}^{n_j}d_{\xi_r}^* + (1/m_j)\sum_{r=1}^{n_j}b_{i_r}^*\| \leqslant 1$, which
implies that $\|\sum_{r=1}^{n_j}f_{i_r}^*\|\leqslant m_j$. We conclude that $(f_i)_i$ cannot be equivalent to the basis of $\ell_1$, a contradiction.
\end{proof}

\subsection{Exact vectors and exact pairs}
An exact pair is a pair of the form  $(\ga,x)$, where $\ga\in\Ga$ and $x\in\X$, which satisfies certain properties. The terms of dependent sequences, which we study in the next subsection, are such pairs. The coordinate $x$ of an exact pair is an exact vector. Its definition below is worth comparing to \cite[Definition 6.1]{AH}, as the constraints appear in the present case.

\begin{dfn}\label{exact vector def}
Let $C\geqslant 1$,  and $j\inn$. A finitely supported vector $x$ in $\X$ is called a $(C,j)$-exact vector if
 \begin{itemize}

 \item[(i)] $\sup_{\eta\in\Ga}|d_\eta^*(x)| \leqslant Cm_j/n_j$,

 \item[(ii)] $\|x\| \leqslant C$  and

 \item[(iii)] for every $i>j$, $\eta\in\Ga$ with $\we(\eta) = m_i^{-1}$ and interval $F$ of $N$,
 \begin{equation*}
|e_\eta^*\circ P_E(x)| <  \frac{C}{m_j},
 \end{equation*}

 \item[(iv)] for every $i<j$, $a\leqslant n_i$  and very fast growing sequence of \ac-averages $(b_r^*)_{r=1}$,
 \begin{equation*}
 \sum_{r=1}^a|b_r^*(x)| < \frac{C}{s(b_1^*)} + \frac{Cm_i}{m_j}
 \end{equation*}
 and

 \item[(v)] $\min\supp x \geqslant m_j$.

 \end{itemize}
\end{dfn}

\begin{dfn}\label{exact pair def}
 Let $C\geqslant 1$, $\theta > 0$ and $j\inn$. A pair $(\ga, x)$ is called a $(C,j,\theta)$-exact pair, if $\ga\in\Ga$ with $\we(\ga) = m_j^{-1}$, $x$ is a $(C,j)$-exact vector and $e_\ga^*(x) = \theta$.
\end{dfn}

\begin{rmk}\label{on exact vectors biorthogonals and index zero}
If $(x_k)_k$ is a block sequence so that each $x_k$ is a $(C,j_k)$-exact vector with $(j_k)_k$ strictly increasing, then (i) and (iv) of Definition \ref{exact vector def} easily imply that $\lim_k\sup_{\ga\in\Ga}|d_{\ga}^*(x_k)| = 0$ and $\adxk = 0$.
\end{rmk}

The following estimate is similar to \cite[Proposition 2.5]{AM1} in a mixed-Tsirelson setting.

\begin{lem}\label{like scc in tsirelson}
Let  $j\inn$ and $k\leqslant n_j$. Then
\begin{equation}
\left\|\frac{m_j}{n_j}\sum_{i=1}^ke_i\right\|_{T[(\mathscr{A}_{4n_j},m_j^{-1})_j]} \leqslant \frac{k}{n_j} + \frac{1}{m_j},
\end{equation}
where the  the norm is taken in $T[(\mathscr{A}_{4n_j},m_j^{-1})_j]$.
\end{lem}

\begin{proof}
Let $x = (m_j/n_j)\sum_{i=1}^ke_i$ and $f$ be a functional in the norming set of $W[(\mathscr{A}_{4n_i},m_i^{-1})_i]$. Define $E_1 = \{i: |f(e_i)|\leqslant 1/m_j\}$, $f_1 = E_1f$ and $f_2 = f - f_1$. Clearly, $|f_1(x)|\leqslant k/n_j$. One can also verify that $f_2\in W[(\mathscr{A}_{4n_i},m_i^{-1})_{i\neq j}]$. By the last statement of \cite[Proposition 2.5]{AH} we obtain $|f_2(x)| \leqslant 1/m_j$.
\end{proof}

The following estimate is a refinement of Lemma \ref{on averages adx zero} for rapidly increasing sequences. It is based on \eqref{ris dominated} and Lemma \ref{like scc in tsirelson}. Its proof is very similar to that of \cite[Lemma 3.7]{AM1}, however we include it for completeness.

\begin{lem}
\label{estimate of average on ris}
Let $j\inn$, $(x_k)_{k=1}^{n_j}$ be a $C$-RIS, $x = (m_j/n_j)\sum_{k=1}^{n_j}x_k$ and $b^*$ be a \ac-average. If $K = \#\{k: \ran b^*\cap\ran x_k\neq\varnothing\}$, then
\begin{equation*}
\left|b^*(x)\right| < \min\left\{m_j\frac{CK/n_j}{s(b^*)}, \frac{10CK/n_j}{s(b^*)} + \frac{10C}{m_j}\right\} + 8C\frac{m_j}{n_j}.
\end{equation*}
\end{lem}

\begin{proof}
If $b^* = (1/p)\sum_{i=1}^d\e_ie_{\ga_i}^*\circ P_{E_i}$ with $1\leqslant d\leqslant p$, define $G = \{k: \ran b^*\cap \ran x_k\neq\varnothing\}$,
\begin{align*}
A_1 & = \{k\in G:\text{ there exists }1\leqslant i\leqslant d\text{ with }\ran x_k\subset E_i\},\\
A_2 & = G\setminus A_1,\text{ and for each }k\in A_2\text{ set}\\
J_k & = \{1\leqslant i\leqslant d: E_i\cap \ran x_k\neq\varnothing\}.
\end{align*}
It is easy to see that $\sum_{k\in A_2}(\#J_k) \leqslant 2\#(\cup_{k\in A_2}J_k) \leqslant 2p$, hence we obtain
\begin{equation}
\label{estimate of average on ris eq1}
\begin{split}
\left|b^*\left(\frac{m_j}{n_j}\sum_{k\in A_2}x_k\right)\right| &\leqslant \frac{m_j}{pn_j}\sum_{k\in A_2}\sum_{i\in J_k}\left|e^*_{\ga_i}\circ P_{E_i}\left(x_k\right)\right|\\
& \leqslant\frac{m_j}{pn_j}\sum_{k\in A_2}\sum_{i\in J_k}\left\|P_{E_k}\right\|\left\|x_k\right\| \leqslant \frac{4Cm_j}{pn_j}\sum_{k\in A_2}(\#J_k)\\
&\leqslant \frac{8Cm_j}{n_j}.
\end{split}
\end{equation}
We used that $\sup\{\|P_E\|: E\text{ is an interval of }\N\}\leqslant 2$.

We now estimate the action of $b^*$ on $A_1$. Define $S = \{1\leqslant i\leqslant d:\text{ there exists }k\in A_1\text{ with }\ran x_k\subset E_i\}$ and for $i\in S$ define $S_i = \{k\in A_1: \ran x_k\subset E_i\}$. Note that $(S_i)_{i\in S}$ defines a partition of $A_1$ into disjoint sets. We evaluate
\begin{align}
\left|b^*\left(\frac{m_j}{n_j}\sum_{k\in A_1}x_k\right)\right| &= \frac{1}{p}\left|\sum_{i\in S}e_{\ga_i}^*\circ P_{E_i}\left(\frac{m_j}{n_j}\sum_{k\in S_i}x_k\right)\right|\nonumber\\
&= \frac{1}{p}\left|\sum_{i\in S}e_{\ga_i}^*\left(\frac{m_j}{n_j}\sum_{k\in S_i}x_k\right)\right|\leqslant \frac{1}{p}\sum_{i\in S}\left\|\left(\frac{m_j}{n_j}\sum_{k\in S_i}x_k\right)\right\|.\label{estimate of average on ris eq2}
\end{align}
We shall treat \eqref{estimate of average on ris eq2} in two different ways. We first just apply the triangle inequality to obtain
\begin{equation}
\label{estimate of average on ris eq3}
\begin{split}
\left|b^*\left(\frac{m_j}{n_j}\sum_{k\in A_1}x_k\right)\right| &\leqslant \frac{1}{p}\sum_{i\in S}(\#S_i)\frac{m_j}{n_j}C = \frac{Cm_j}{pn_j}\#A_1\leqslant \frac{Cm_j}{pn_j}K\\
& = m_j\frac{C(K/n_j)}{s(b^*)}.
\end{split}
\end{equation}
The second way is to apply \eqref{ris dominated} to \eqref{estimate of average on ris eq2} and then combine the result with Lemma \ref{like scc in tsirelson} as follows:
\begin{align}
\left|b^*\left(\frac{m_j}{n_j}\sum_{k\in A_1}x_k\right)\right| &\leqslant \frac{1}{p}\sum_{i\in S}10C\left\|\frac{m_j}{n_j}\sum_{k\in S_i}e_k\right\|_{T[(\mathscr{A}_{4n_j},m_j^{-1})_j]}\nonumber\\
&\leqslant \frac{10C}{p}\sum_{i\in S}\left(\frac{\#S_i}{n_j} + \frac{1}{m_j}\right)\nonumber\\
& = \frac{10C}{s(b^*)}\frac{\sum_{i\in S}\#S_i}{n_j} + \frac{10C}{m_j}\frac{\#S}{p}\leqslant \frac{10C}{s(b^*)}\frac{\#A_1}{n_j} + \frac{10C}{m_j}\frac{d}{p}\nonumber\\
&\leqslant \frac{10C(K/n_j)}{s(b^*)} + \frac{10C}{m_j}.\label{estimate of average on ris eq4}
\end{align}
We combine \eqref{estimate of average on ris eq1}, \eqref{estimate of average on ris eq3}, and \eqref{estimate of average on ris eq4} to obtain the desired conclusion.
\end{proof}

\begin{lem}\label{this is just a small part of the next one}
Let $j > i$ be natural numbers, $(x_k)_{k=1}^{n_j}$ be a $C$-RIS with $\min\supp x_k \geqslant j-1$ and set $x = (m_j/n_j)\sum_{k=1}^{n_j}x_k$. Let moreover $(b_r^*)_{r=1}^a$ be a very fast growing sequence of \ac-averages, with $a\leqslant n_{i}$ and assume that for every $1\leqslant r \leqslant a$ there is at most one $1\leqslant k\leqslant n_j$ so that $\ran b^*\cap\ran x_k\neq\varnothing$. Then
\begin{equation*}
\sum_{r=1}^a|b_r^*(x)| < \frac{24Cm_i}{m_{j}}.
\end{equation*}
\end{lem}

\begin{proof}
Set $G = \{k:\text{there is}\;r\;\text{with}\;\ran b_r\cap\ran x_k\neq\varnothing\}$ and $x' = (m_j/n_j)\sum_{k\in G}x_k$. Note that $\#G\leqslant a\leqslant n_i$. By changing the signs of the $b_r^*$ and, perhaps, omitting some of the first few terms, we may assume that $\max\supp b_1^* \geqslant \min\supp x_1\geqslant j$ and that $|b_r^*(x)|  = b_r^*(x')$ for $r=1,\ldots,a$. Proposition \ref{build on vfg} implies that there are $\ga$ and $(\xi_r)_{r=1}^a$ in $\Ga$, so that
\begin{equation}
\frac{1}{m_i}\sum_{r=1}^a|b_{r}^*(x)| = \frac{1}{m_i}\sum_{r=1}^a b_r^*\left(x'\right) = e_\ga^*\left(x'\right) - \sum_{r=1}^ad_{\xi_r}^*\left(x'\right)
\end{equation}
Lemma \ref{like scc in tsirelson} and \eqref{ris dominated} imply that $|e_\ga^*(x')|\leqslant 10Cn_i/n_j + 10C/m_j$ while it easily follows that $\sum_{r=1}^a|d_{\xi_r}^*(x')| \leqslant 4Cn_im_j/n_j$. We obtain $\sum_{r=1}^a|b_r^*(x)| < 10Cn_im_i/n_j + 10Cm_i/m_j + 4Cn_im_im_j/n_j$. The choice of the sequences $(m_j)_j$, $(n_j)_j$ yields the desired estimate.
\end{proof}

The following Lemma is proved using Lemmas \ref{estimate of average on ris} and \ref{this is just a small part of the next one} and arguments very similar to those used in the proof of \cite[Lemma 3.8]{AM1}. We include a proof for completeness.

\begin{lem}\label{this yields that on ris one can construct exact vector}
Let $j>i$ be natural numbers, $(x_k)_{k=1}^{n_j}$ be a $C$-RIS with $\min\supp x_k \geqslant m_jn_j$ and set $x = (m_j/n_j)\sum_{k=1}^{n_j}x_k$. Let moreover $(b_r^*)_{r=1}^a$ be a very fast growing sequence of \ac-averages, with $a\leqslant n_{i}$. Then
\begin{equation*}
\sum_{r=1}^a|b_r^*(x)| < \frac{10C}{s(b_1^*)} + \frac{50Cm_i}{m_j}.
\end{equation*}
\end{lem}

\begin{proof}
We define $r_1$ to be the minimum $r$ for which $\ran b_r^*\cap\ran x\neq\varnothing$. Define the sets $R_1 = \{r>r_1:\text{ there is at most one } k \text{ with }\ran b_r^*\cap\ran x_k\neq\varnothing\}$ and $R_2 = \{r_1+1,\ldots,a\}\setminus R_1$. By Lemma  \ref{this is just a small part of the next one} we obtain
\begin{equation}
\label{this yields that on ris one can construct exact vector eq1}
\sum_{r\in R_1}|b_r^*(x)| < \frac{24Cm_i}{m_j}.
\end{equation}
On the other hand, for $r\in R_2$, by Lemma \ref{estimate of average on ris} we obtain the estimate $|b_r^*(x)| < m_jC/s(b_r^*) + 8Cm_j/n_j < m_jC/2^{\min\supp x} + 8Cm_j/n_j$ and hence
\begin{equation}
\label{this yields that on ris one can construct exact vector eq2}
\sum_{r\in R_2}|b_r^*(x)| < n_i\frac{10Cm_j}{2^{m_jn_j}} + n_i\frac{8Cm_j}{n_j}.
\end{equation}
For $r=1$ by Lemma \ref{estimate of average on ris} we obtain
\begin{equation}
\label{this yields that on ris one can construct exact vector eq3}
|b_{r_1}^*(x)| < \frac{10C}{s(b_{r_1}^*)} + \frac{10C}{m_j} + \frac{8Cm_j}{n_j}.
\end{equation}
We obtain the result by combining \eqref{this yields that on ris one can construct exact vector eq1}, \eqref{this yields that on ris one can construct exact vector eq2}, and \eqref{this yields that on ris one can construct exact vector eq3} and using the lacunarity properties of $(m_j,n_j)_j$
\end{proof}

\begin{prp}\label{on RIS are exact vectors and pairs}
Let $j\inn$ and  $(x_k)_{k=1}^{n_j}$ be a $C$-RIS with $\min\supp x_1 \geqslant m_jn_j$. Then $x = (m_j/n_j)\sum_{k=1}^{n_j}x_k$ is a $(112C,j)$-exact vector. If moreover $(x_k)_{k=1}^{n_j}$ is skipped, there are $\theta > 0$ and a very fast growing sequence of \ac-averages $(b_k^*)_{k=1}^{n_j}$ so that $b_k^*(x_k)_k = \theta$ for $k=1,\ldots,n_j$, then there is $\ga\in\Ga$ so that $(\ga,x)$ is a $(112C,j,\theta)$-exact pair.
\end{prp}

\begin{proof}
The first part follows from \eqref{currant}, \eqref{raspberries} and Lemma \ref{this yields that on ris one can construct exact vector}, while the second part follows from Proposition \ref{build on vfg}.
\end{proof}

\subsection{Dependent sequences} We finally define dependent sequences an describe how they can be found in every block subspace. Their definition is based on the tree $\mathcal{U}$ of special sequences (see Subsection \ref{subsection tree}). Note that $\mathcal{U}$ was defined using the space $\BmT$. Here, we naturally identify finitely supported vectors in $\X$ with ones in $\BmT$.

\begin{ntt}
For a finitely supported vector $x = \sum_{\ga\in\Ga}\la_\ga d_\ga$ in $\X$, we denote by $\bar{x}$ the vector $\sum_{\ga\in\Ga}\la_\ga \bar{d}_\ga$ in $\BmT$.
\end{ntt}

\begin{rmk}\label{forget those bars}
Remark \ref{bars cancel out and stuff} yields that if $\ga\in\Ga$ and $E$ is an interval of $\N$, then $e_\ga^*\circ \bar P_{E}(\bar{x}) = e_\ga^*\circ P_E(x)$.
\end{rmk}

\begin{dfn}\label{dependent sequence def}
Let $C\geqslant 0$ and $\theta > 0$. A sequence of pairs $\{(\ga_k,x_k)\}_{k=1}^\ell$, where $\ga_k\in\Ga$ and $x_k$ is a finitely supported vector of $(d_\ga)_{\ga\in\Ga}$ with rational coefficients for $k=1,\ldots,\ell$, is called a $(C,\theta)$-dependent sequence, if the following are satisfied:
\begin{itemize}

\item[(i)] $(\ga_k,x_k)$ is  a $(C,j_k,\theta)$-exact pair, where $\we(\ga_k) = m_{j_k}^{-1}$,

\item[(ii)] $\{(\ga_k, \bar{x}_k)\}_{k=1}^{\ell}$ is a special sequence (i.e. it is in $\mathcal{U}$) and

\item[(iii)] $\min\supp x_{k+1} > \max\{\max_{i\leqslant k} \ra(\ga_i), \max\supp x_k\}$ for $k<\ell$.

\end{itemize}
An infinite sequence of pairs $\{(\ga_k,x_k)\}_{k=1}^\infty$, so that for each $\ell\inn$ the first $\ell$-terms $\{(\ga_k,x_k)\}_{k=1}^\ell$ define a $(C,\theta)$-dependent sequence, will be called a $(C,\theta)$-dependent sequence as well.
\end{dfn}

\begin{rmk}
Let $0 < C\leqslant 16000$, $\theta > 0$ and $\{(\ga_k,x_k)\}_{k=1}^\ell$ be a $(C,\theta)$-dependent sequence. If $E_k = \ran x_k$ for $k=1,\ldots,\ell$, then Remark \ref{forget those bars} easily implies that $\{(\ga_k,E_k)\}_{k=1}^\ell$ is comparable, in the sense of Definition \ref{def averages}.
\end{rmk}

\begin{prp}\label{exact pairs exist}
Let $(x_k)_k$ be a normalized block sequence with rational coefficients in $\X$, which satisfies $\adxk = 0$ and $\lim_k \sup_{\ga\in\Ga} |d_\ga^*(x_k)| = 0$. Let also $(\eta_k)_k$ be a sequence in $\Ga$ so that $|e_{\eta_k}^*(x_k)| > (3/4)\|x_k\|$ for all $k\inn$. Then for every $j\inn$ there exists a $(3584,j,1)$-exact pair $(\ga,y)$, so that,
\begin{equation}\label{buinlding blocks look like this}
\begin{split}
y& = \frac{m_{j}}{n_{j}}\sum_{r=1}^{n_{j}}\sum_{i\in F_{r}}\la_i\e_ix_i\;\text{and}\\
e_{\ga}^*& = \frac{1}{m_{j}}\sum_{r=1}^{n_{j}}\frac{1}{\#F_{r}}\sum_{i\in F_{r}}\e_ie_{\eta_i}^*\circ P_{E_i} + \sum_{r=1}^{n_{j}}d_{\xi_{r}}^*
\end{split}
\end{equation}
where $0<|\la_i| <4/3$, actually $\la_i = 1/e^*_{\ga_i}(x_i)$, $(F_{r})_{r}$ are successive subsets in $\mathcal{S}_1$, $(\e_i)_{i\in F_{r}}$ are alternating signs for $r\inn$ and $\ra(d_{\xi_{r}}^*)\notin \ran x_m$ for all $r$, $m$.
\end{prp}

\begin{proof}
Pass to a subsequence satisfying the assumption of Remark \ref{first step building better} and define $\la_k = 1/e^*_{\ga_k}(x_k)$ for all $k\inn$. The conclusion follows from Proposition \ref{adx zero} then applying Proposition \ref{on RIS are exact vectors and pairs}.
\end{proof}

\begin{rmk}\label{dependent sequences exist}
Proposition \ref{c0 sm in every subspace} and Proposition \ref{exact pairs exist} immediately yield that, up to a perturbation, in every infinite dimensional subspace $X$ of $\X$, there is a $(3584,1)$-dependent sequence $\{(\ga_k,x_k)\}_{k=1}^\infty$, so that $x_{k}\in X$  for all $k\inn$.
\end{rmk}

Using Proposition \ref{basis c0}, the following result can be shown, where the sequence $(x_k)_k$ satisfying $\adxk = 0$ and $\lim_k \sup_{\ga\in\Ga} |d_\ga^*(x_k)| = 0$ can be replaced with a subsequence of the basis.

\begin{prp}\label{dependent sequence on subsequence of the basis}
Let $(d_{\ga_i})$ be a subsequence of the basis of $\X$. Then there exist $\theta > 0$ and a $(3584,\theta)$-dependent sequence $\{(\ga_k,y_k')\}_{k=1}^\infty$, so that for each $k$,
\begin{equation}\label{capital control for the win}
\begin{split}
y_k& = \frac{m_{j_k}}{n_{j_k}}\sum_{r=1}^{n_{j_k}}\theta\sum_{i\in F_{r,k}} \e_id_{\ga_i}\;\text{and}\\
e_{\ga_k'}^*& = \frac{1}{m_{j_k}}\sum_{r=1}^{n_{j_k}}\frac{1}{\#F_{r,k}}\sum_{i\in F_{r,k}} \e_id^*_{\ga_i} + \sum_{r=1}^{n_{j_k}}d_{\xi_{r,k}}^*
\end{split}
\end{equation}
where $(F_{r,k})_{r,k}$ are successive subsets in $\mathcal{S}_1$ and $(\e_i)_{i\in F_{r}}$ are alternating signs for $r\inn$.
\end{prp}

\section{Estimations on dependent sequences}

In the previous section we defined dependent sequences and proved their existence in every block subspace. In this section we provide an estimate for the norm of finite sums of consecutive terms of such sequences, which yields that the space $\X$ contains no boundedly complete sequence and that it is hereditarily indecomposable. We also observe that every subspace of $\X$ fails the PCP and hence also the RNP.

\begin{lem}\label{components of averages practicaly just estimate larger weights than theirselves}
Let $1\leqslant C\leqslant 4000$, $\theta > 0$, $\{(\ga_k,x_k)\}_{k=1}^\ell$ be a $(C,\theta)$-dependent sequence and $1\leqslant n\leqslant m \leqslant \ell$ be natural numbers. Let also $(\eta_j)_{j=1}^d$ be a sequence in $\Ga$, $(E_j)_{j=1}^d$ be a sequence of intervals of $\N$ and $(\e_j)_{j=1}^d$ be a sequence of signs, so that one of the following is satisfied:
\begin{itemize}

\item[(i)] the sequence $(\eta_j, E_j)_{j=1}^d$ is comparable and the signs $(\e_j)_{j=1}^d$ are alternating or

\item[(ii)] the sequence $(\eta_j, E_j)_{j=1}^d$ is either incomparable or irrelevant.

\end{itemize}
If for $j=1,\ldots,d$ we define $D_j = \{n\leqslant k\leqslant m: \we(\ga_k) < \we(\eta_j)\}$, then
\begin{equation*}
\begin{split}
\left|\sum_{j=1}^d\e_je_{\eta_j}^*\circ P_{E_j}\left(\sum_{k=n}^mx_k\right) - \sum_{j=1}^d\e_je_{\eta_j}^*\circ P_{E_j}\left(\sum_{k\in D_j}x_k\right)\right|\\
\leqslant 9C + 2dC\we(\ga_n).
\end{split}
\end{equation*}
\end{lem}

\begin{proof}
The goal of this Lemma is to show that the two sums in the absolute value of the statement are sufficiently close to each other. The concept behind the proof is to compare the weights of the $\ga_k$'s and the weights of the $\eta_j$'s. It is first shown that the action of all $e_{\eta_j}^*$'s on all $x_k$'s that have different weights is negligible. This is done using standard techniques from HI constructions. On the other other hand, to show that the action of $e_{\eta_j}^*$'s with weights equal to some of those of the $\ga_k$'s it is necessary to use the notion of comparable sequences (Definition \ref{def averages} (ii)). Actually, the treatment of this case is the main reason for introducing this notion.

For $k=1,\ldots,\ell$ we define $A_k = \{j: \we(\eta_j) = \we(\ga_k)\}$, $B_k = \{j: \we(\eta_j) > \we(\ga_k)\}$ and $C_k = \{j: \we(\eta_j) < \we(\ga_k)\}$. By Definition \ref{exact vector def} (iii), for $k=1,\ldots,m$ we obtain
\begin{equation}\label{gnrsioghaorkoeaw}
\left|\sum_{j\in C_k}\e_je_{\eta_j}^*\circ P_{E_j}(x_k)\right| < C\we(\ga_k).
\end{equation}
Observe that
$$\sum_{j=1}^d\e_je_{\eta_j}^*\circ P_{E_j}\left(\sum_{k\in D_j}x_k\right) = \sum_{k=n}^m\sum_{j\in B_k}\e_je_{\eta_j}^*\circ P_{E_j}(x_k),$$
therefore
\begin{eqnarray*}
\left|\sum_{j=1}^d\e_je_{\eta_j}^*\circ P_{E_j}\left(\sum_{k=n}^mx_k\right) - \sum_{j=1}^d\e_je_{\eta_j}^*\circ P_{E_j}\left(\sum_{k\in D_j}x_k\right)\right| =\\
\left|\sum_{k=n}^m\sum_{j\in A_k}\e_je_{\eta_j}^*\circ P_{E_j}(x_k) + \sum_{k=n}^m\sum_{j\in C_k}\e_je_{\eta_j}^*\circ P_{E_j}(x_k)\right|\leqslant\\
\left|\sum_{k=n}^m\sum_{j\in A_k}\e_je_{\eta_j}^*\circ P_{E_j}(x_k)\right| + Cd\sum_{k=n}^m\we(\ga_k),
\end{eqnarray*}
where the inequality follows \eqref{gnrsioghaorkoeaw}. By the choice of the sequence $(m_j)_j$, we obtain $\sum_{k=n}^m\we(\ga_k) \leqslant 2\we(\ga_n)$. Hence, all that remains to be shown is $|\sum_{k=n}^m\sum_{j\in A_k}\e_je_{\eta_j}^*\circ P_{E_j}(x_k)| \leqslant 9C$.

Observe that each set $A_k$ is either empty or a singleton and in particular, we note that $j\in A_k$ if and only if $\we(\eta_j) = \we(\ga_k)$. If the sets $A_k$ are all empty there is nothing to prove. Otherwise, let $k_1 < \cdots <k_s$ be all the $k$'s in $\{1,\ldots,\ell\}$ satisfying $A_{k_i} \neq \varnothing$. Let also $1\leqslant j_1 < \cdots < j_s\leqslant d$ be so that for each $i$, $j_i$ is the unique element of $A_{k_i}$, and hence $\we(\eta_{j_i}) = \we(\ga_{k_i})$  for $i=1,\ldots,s$.

Is $s\leqslant 2$, then the desired estimate follows from $\|x_k\| \leqslant C$ and $\|e_{\eta_i}^*\circ P_{E_i}\|\leqslant 4$. Otherwise, $s\geqslant 3$ which implies that the sequence $(\eta_j,E_j)_{j=1}^d$ (for a detailed argument see the proof of \cite[Lemma 5.8]{AM2}).

We conclude that the sequence $(\eta_j,E_j)_{j=1}^d$ is either comparable, or irrelevant and therefore there exists $m'\inn$ with $d\leqslant m'$, natural numbers $1\leqslant k_1' <\cdots < k_d'\leqslant m'$ and $\{(\xi_k, y_k)\}_{k=1}^{m'}$ in $\mathcal{U}$, so that $\we(\eta_j) = \we(\xi_{k_j'})$ for $j=1,\ldots,d$. It follows that
\begin{itemize}

\item[(a)] $j_i = i$ for $i=1,\ldots,s$,

\item[(b)] $k_i' = k_i$ for $i=1,\ldots,s$ and

\item[(c)] $\xi_k = \ga_k$,  $y_k = \bar{x}_k$ for $k=1,\ldots,k_s-1$.

\end{itemize}
For a detailed argument explaining the above see once more \cite[Lemma 5.8]{AM2}. We observe that the sequence is not irrelevant. Indeed, the opposite would imply $16000 <|e_{\eta_2}^*\circ \bar{P}_{E_2}(\bar{x}_{k_2})| = |e_{\eta_2}^*\circ {P}_{E_2}({x}_{k_2})| \leqslant 4C\leqslant 16000$, where the equality follows from Remark \ref{forget those bars} and the first inequality from $\|x_k\| \leqslant C$ and $\|e_{\eta_2}^*\circ P_{E_2}\| \leqslant 4$.

We have therefore shown that the sequence $(\eta_j,E_j)_{j=1}^d$ is comparable. Define $J = \{i:\;k_i\in\{n,\ldots,m\}\}$, observe that $J$ is an interval of $\{1,\ldots,s\}$ and choose successive two-point intervals $J_1,\ldots,J_p$ of $J\setminus\{\max J,\min J\}$, so that $J\setminus\cup_{i=1}^p{J_i}$ has at most three elements. The fact that the sequence $(\eta_j,E_j)_{j=1}^d$ is comparable yields
\begin{equation}\label{rhubarb}
\sum_{i=1}^p\left|\sum_{j\in J_i}\e_je_{\eta_j}^*\circ P_{E_j}(x_{k_j})\right| \leqslant 1
\end{equation}
and hence
\begin{equation*}
\left| \sum_{k=n}^m\sum_{j\in A_k}\e_je_{\eta_j}^*\circ P_{E_j}(x_k)\right| = \left|\sum_{j\in J}\e_je_{\eta_j}^*\circ P_{E_j}(x_{k_j})\right| \leqslant 8C + 1 \leqslant 9C.
\end{equation*}
For a more detailed explanation of \eqref{rhubarb} see the proof of \cite[Lemma 5.8]{AM2}.
\end{proof}

The following result is the main estimate of this section.

\begin{prp}\label{estimate on dependent sequence}
Let $1\leqslant C\leqslant 4000$, $\theta > 0$, and $\{(\ga_k,x_k)\}_{k=1}^\ell$ be a $(C,\theta)$-dependent sequence.
\begin{itemize}

\item[(i)] Let $\ga\in\Ga$ and $E$ be an interval of $\N$. If for some natural numbers $1\leqslant n\leqslant m\leqslant \ell$ we set $D = \{k\in[n,m]: \we(\ga_k) < \we(\ga)\}$, then:
\begin{equation*}
\left|e_{\ga}^*\circ P_E\left(\sum_{k\in D} x_k\right)\right| \leqslant 63C\we(\ga).
\end{equation*}

\item[(ii)] For all natural numbers $1\leqslant n\leqslant m\leqslant \ell$ we have
$$\left\|\sum_{k=n}^mx_k\right\| \leqslant 10 C.$$

\end{itemize}
\end{prp}

\begin{proof}
We use a standard argument to prove the second statement given the first one. Let $1\leqslant n\leqslant m\leqslant \ell$ and $\ga\in\Ga$ Define $k_0 = \min \{n\leqslant k \leqslant m: \we(\ga_k) \leqslant \we(\ga)\}$ (if $k_0$ is not well defined, then the result follows from a similar argument as the one we shall use in this case). We conclude that $|e_\ga^*(\sum_{k=k_0+1}^mx_k)| \leqslant 63C\we(\ga)$ and $|e_\ga^*(x_{k_0})| \leqslant C$. To compute the action on the rest of the vectors we use the evaluation analysis of $\ga$, $e_\ga^* = \we(\ga)\sum_{r=1}^ab_r^* + \sum_{r=1}^ad_{\xi_r}^*$. If $\we(\ga) = m_{j_0}^{-1}$, then $\we(\xi_r) = m_{j_0}^{-1}$ as well, which yields $\ra(\xi_r) \geqslant j_0$ for $r = 1,\ldots,a$. Recall that $\{(\ga_k,\bar{x}_k)\}_{k=1}^\ell$, which in conjunction with \eqref{coding growth} implies that $j_0 > \max\supp x_{k_0-1}\we(\ga_{k_0 -1})$. We obtain
$$\left|e_\ga^*\left(\sum_{k=n}^{k_0-1}x_k\right)\right| = \we(\ga)\left|b_1^*\left(\sum_{k=n}^{k_0-1}x_k\right)\right| \leqslant \we(\ga)4Ck_0 \leqslant C.$$
Combining the above the conclusion follows.

We now proceed to prove the first statement by induction on the rank of $\ga$. The case $\ra(\ga) = 1$ is easy, so let $p\inn$ such that for every $\ga\in\Ga$ with $\ra(\ga)\leqslant p$ and interval $E$ of $\N$ the conclusion is satisfied.

We remark the following: let $b^*$ be an \ac-average of $B_{0,p}$ and $n\leqslant m$, then
\begin{equation}\label{raspberry}
\left|b^*\left(\sum_{k=n}^mx_k\right)\right| \leqslant \frac{30C}{s(b^*)} + 2C\we(\ga_n).
\end{equation}
This follows from Definition \ref{exact vector def} (i) in case $b^*$ is a basic average and from the inductive assumption and Lemma \ref{components of averages practicaly just estimate larger weights than theirselves} and the choice of $(m_j)_j$ otherwise.

Let now $\ga\in\Ga$ with $\ra(\ga) = p+1$ and $E$ be an interval of $\N$. Let $e_\ga^*  = (1/m_j)\sum_{r=1}^ab_r^* + \sum_{r=1}^ad_{\xi_r}^*$, with $a\leqslant n_j$, be the evaluation analysis of $\ga$ and note that $e_\ga^*  = (1/m_j)\sum_{r=1}^a\bar{b}_r^{*} + \sum_{\ra(\xi_r)\in E}d_{\xi_r}^*$, where $\bar{b}_r^{*} = b_r^*\circ P_E$. We note that $(\bar{b}_r^*)_{r=1}^a$ is a very fast growing sequence of \ac-averages. Let $1\leqslant n\leqslant m\leqslant \ell$, set $D = \{k\in[n,m]: \we(\ga_k) < \we(\ga)\}$ and for $k\in D$ set $M_k = \{r: \ran \bar{b}^*_r\cap\ran x_k\neq\varnothing\}$, $N_k = \{r\in M_k: s(\bar{b}_r^*) \geqslant (\we(\ga_k))^{-1}\}$. By Definition \ref{exact vector def} (iv) we obtain
\begin{equation}\label{passion fruit}
\left|\sum_{r=1}^a\bar{b}_r^*\left(\sum_{k\in D}x_k\right)\right|\leqslant \left|\sum_{k\in D}\sum_{q\in M_k\setminus N_k}\bar{b}_r^*(x_k)\right| + 2C\sum_{k\in D}\frac{\we(\ga_k)}{\we(\ga)}.
\end{equation}

Define $A = \cup_{k\in D} M_k\setminus N_k$, for $r\in A$ set $D_r = \{k\in D:\;r\in M_k\setminus N_k\}$ and observe that
\begin{equation}\label{donut}
\left|\sum_{k\in D}\sum_{r\in M_k\setminus N_k}\bar{b}_r^*(x_k)\right| = \left|\sum_{r\in A}\bar{b}_r^*\left(\sum_{k\in D_r}x_k\right)\right|.
\end{equation}
Following the arguments used in the proof of \cite[Proposition 5.9]{AM2} and using Definition \ref{exact vector def} (iv), we conclude that the $D_r$'s are disjoint intervals of $\{n,\ldots,m\}$. We set $n_r = \min D_r$ and note that the $n_r$'s are strictly increasing. By \eqref{raspberry} we obtain
\begin{equation}\label{scissors}
\left|\sum_{r\in A}\bar{b}_r^*\left(\sum_{k\in D_r}x_k\right)\right| \leqslant \sum_{r\in A}\left(\frac{30C}{s(\bar{b}_r^*)} + 2C\we(\ga_{n_r})\right).
\end{equation}
The fact that $(\bar{b}_r^*)_{r=1}^a$ is very fast growing, the choice of the sequence $(m_j)_j$ and \eqref{passion fruit}, \eqref{donut}, \eqref{scissors} yield
\begin{equation}\label{sausage}
\left|\sum_{r=1}^a\bar{b}_r^*\left(\sum_{k\in D}x_k\right)\right|\leqslant 62C
\end{equation}
whereas Definition \ref{exact vector def} (i) and the choice of $(m_j)_j$ imply
\begin{equation}\label{sauerkraut}
\sum_{\ra(\xi_r)\in E}d_{\xi_r}^*(\sum_{k\in D}x_k) \leqslant C\we(\ga).
\end{equation}
Combining \eqref{sausage} and \eqref{sauerkraut} we conclude $|e_\ga^*\circ P_E(\sum_{k\in D}x_k)| \leqslant 63C\we(\ga)$.
\end{proof}

Remark \ref{dependent sequences exist} and Proposition \ref{estimate on dependent sequence} (ii) yield that $\X$ contains no boundedly complete sequence, which yields the following.

\begin{thm}\label{no reflexive subspace}
The space $\X$ contains no reflexive subspace.
\end{thm}

\begin{prp}
\label{differences on dependent sequences blow up}
Let $1\leqslant C\leqslant 4000$, $\theta > 0$, and $\{(\ga_k,x_k)\}_{k=1}^\ell$ be a $(C,\theta)$-dependent sequence. Then
$$\sup_n\left\|\sum_{k=1}^n(-1)^kx_k\right\| = \infty.$$
\end{prp}

\begin{proof}
Inductively choose successive a sequence $(E_r)_r$ of finite intervals  of $\N$ so that $\cup_r E_r = \N$, for each $r\in\N$
$$b_r^* = \frac{1}{\#E_n}\sum_{k\in E_n}(-1)^ke_{\ga_k}^*\circ P_{\ran x_k}$$
is an \ac-average, and the sequence $(b_r^*)_r$ is very fast growing. Note that for $n\geqslant\max E_r$ we have $b_r^*(\sum_{k=1}^n(-1)^kx_k) = \theta$. For $j\in\N$ we construct a $\gamma\in\Ga$ with $e_\ga^* = (1/m_{j_0})\sum_{r=r_0+1}^{r_0+n_j}b_r^* + \sum_{r = r_0+1}^{r_0+n_j}d_{\xi_r}^*$, for $r_0\in\N\cup\{0\}$ appropriately large so that $d_{\xi_r}^*(x_k) = 0$ for $r_0 < r\leqslant r_0 + n_j$ and $k\in\N$. Finally, observe that for $k\geqslant E_{r_0+n_j}$ we have $e_\ga^*(\sum_{k=1}^n(-1)^kx_k) = (n_j/m_j)\theta$, which yields the desired result.
\end{proof}

\begin{rmk}
\label{subsets of dependent sequences blow up}
A very similar proof also yields that if  $1\leqslant C\leqslant 4000$, $\theta > 0$, $\{(\ga_k,x_k)\}_{k=1}^\ell$ is a $(C,\theta)$-dependent sequence, and $L$ is an infinite subset of $\N$ with $\N\setminus L$ infinite, then
$$\sup_n\Big\|\sum_{\left\{\substack{k\in L\\k\leqslant n}\right\}}x_k\Big\| = \infty.$$
\end{rmk}

\begin{thm}\label{'tis hi}
The space $\X$ is hereditarily indecomposable.
\end{thm}

\begin{proof}
As it is stated in Remark \ref{dependent sequences exist}, every infinite dimensional subspace contains a perturbation of a $(3584,1)$-dependent sequence $\{(\ga_k,x_k)\}_{k=1}^\infty$. Given two infinite dimensional subspaces $X$ and $Y$, the dependent sequence can be chosen so that for $k$ even $x_k$ is in $X$ and for $k$ odd $x_k$ is in $Y$, at least up to a small enough perturbation. Set $u_n = \sum_{k=1}^nx_{2k}$ and $w_n = \sum_{k=1}^nx_{2k-1}$ for all $n\in\N$. Then $u_n\in X$, $w_n\in Y$ for all $n\inn$ and by Proposition \ref{estimate on dependent sequence} (ii) $\sup_n\|u_n + w_n\| <\infty$,  whereas by Proposition \ref{differences on dependent sequences blow up} $\sup_n\|u_n - w_n\| = \infty$. This yields that $\X$ is hereditarily indecomposable.
\end{proof}

\begin{rmk}
Bourgain posed the question whether there exists a $\mathscr{L}_\infty$-space not containing $c_0$ that also fails the RNP \cite[Problem 3, page 46]{B}. This was answered positively in \cite{FOS}, where is was proved that every Banach space $X$ with separable dual embeds into a $\mathscr{L}_\infty$-space $Y$ with separable dual as well. If $X$ contains no copy of $c_0$, $Y$ can be chosen to contain no copy of $c_0$ either. The space in this paper also provides an answer to Bourgain's question in a strong way. Indeed, by \cite[Theorem 4.1]{EW}, every subspace of $\X$ fails the PCP and hence also the RNP.
\end{rmk}

\section{Operators on the space $\X$}
The goal of this section is to prove that the space $\X$ satisfies the scalar-plus-compact property. We initially characterize strictly singular operators with respect to their behavior on certain sequences generating $c_0$ spreading models. Then, we prove that the space has the scalar-plus-strictly singular property. Finally, we use the aforementioned characterization to deduce that strictly singular operators on $\X$ are compact. Recall that in \cite{AH}, an operator $T:\mathfrak{X}_{\mathrm{AH}}\rightarrow\mathfrak{X}_{\mathrm{AH}}$ is compact if and only if it maps all rapidly increasing sequences to norm-null ones. In the present case, non-compact operators $T:\X\rightarrow\X$ always map some rapidly increasing sequence to a sequence with $\al$-index positive.

Definition \ref{exact pair def} (i) and the fact that the extension operators have norm at most two, yields the following result.

\begin{rmk}\label{some remark that lets you find a coordinate and interval providing the estimate that you more or less require}
Let $(\ga,x)$ be a $(C,j,\theta)$-exact pair and $\rho\in[0,\theta]$. Then there is an interval $E$ of $\ran x$ so that $|e_\ga^*\circ P_E(x) - \rho| < 2Cm_j/n_j$.
\end{rmk}

\begin{prp}\label{ss char}
Let $X$ be an infinite dimensional closed subspace of $\X$ and $T:X\rightarrow\X$ be a bounded linear operator. The following assertions are equivalent.
\begin{itemize}

\item[(i)] The operator $T$ is strictly singular.

\item[(ii)] There is a normalized weakly null sequence $(x_k)_k$ in $X$ so that $(Tx_k)_k$ is norm-null.

\item[(iii)] For every normalized sequence $(x_k)_k$ in $X$ generating a $c_0$ spreading model and $\lim_k\sup_{\ga\in\Ga}|d_\ga^*(x_k)| = 0$, $(Tx_k)_k$ is norm-null.

\end{itemize}
\end{prp}

\begin{proof}
The equivalence of (i) and (ii) is a general property of hereditarily indecomposable Banach spaces, whereas the implication (iii)$\Rightarrow$(i), is an immediate consequence of Proposition \ref{c0 sm in every subspace}. It remains to prove that (ii) implies (iii) and towards a contradiction assume that this is not the case, i.e. there are a normalized weakly null sequence $(z_k)_k$ in $X$ so that $(Tz_k)_k$ is norm-null and a normalized sequence $(x_k)_k$ in $X$ generating a $c_0$ spreading model with $\lim_k\sup_{\ga\in\Ga}|d_\ga^*(x_k)| = 0$, so that $(Tx_k)_k$ is not norm-null. We note that $(Tx_k)_k$ only admits the unit vector basis of $c_0$ as a spreading model.  We apply Proposition \ref{c0 sm in every subspace}, perturb and scale the operator $T$, perhaps defining it on a different subspace $X'$ of $\X$, so that $(x_k)_k$, $(Tx_k)_k$ $(z_k)_k$ are all block sequences with rational coefficients, $\adzk = 0$ as well as  $\lim_k\sup_{\ga\in\Ga}|d_\ga^*(z_k)| = 0$, $\|Tx_k\| = 1$ and $Tz_k = 0$ for all
$k\inn$.

By perhaps changing the signs of some of the sequences or the operator $T$ and passing to subsequences, we may choose three sequences $(\eta_k)_k$, $(\eta_k')_k$, $(\theta_k)_k$ of $\Ga$ so that $\sum_k|e_{\eta_k}^*(x_k) - 1| < \infty$, $\sum_k|e_{\eta_k'}^*(z_k) - 1| < \infty$ and $\sum_k|e_{\theta_k}^*(Tx_k) - 1| < \infty$. We consider two cases, namely whether the set $\{(\we(\theta_k))^{-1}:k\inn\}$ is bounded or not. We shall only treat the second case, as the first one follows using Lemma \ref{tools to build on c0} and the fact that basic averages are always \ac-averages. We therefore assume that the set $\{(\we(\theta_k))^{-1}:k\inn\}$ is unbounded.

We define $E_k = \ran x_k$, $E_k' = \ran z_k$ and $F_k$ to be the smallest interval containing $E_k$ as well as $\ran Tx_k$. By passing to a subsequence we may choose $\rho\in[0,1]$ so that $\sum_{k}|e_{\theta_k}^*\circ P_{F_k}((1/e_{\eta_k}^*(x_k))x_k) - \rho| < \infty$.

Using Proposition \ref{exact pairs exist} we carefully construct a $(3584,1)$-dependent sequence $\{(\ga_k,y_k)\}$, so that for $k$ even $(\ga_k,y_k)$ is built on $(\eta_i)_i$, $(x_i)_i$, whereas for $k$ odd it is built on  $(\eta_i')_i$, $(z_i)_i$, as in \eqref{buinlding blocks look like this}. For $k$ even, $\ga_k$ can be built in such a manner that there is $\zeta_k\in \Ga$ ($\zeta_k$ is built on the $\theta_i$'s) and an interval $J_k$ of $\N$ ($J_k$ can be chosen to be the smallest interval containing the ranges of both $y_k$ and $Ty_k$) with $e_{\zeta_k}^*\circ P_{J_k}(Ty_k) > 1/2$ with $\we(\zeta_k) = \we(\ga_k)$ and $|e_{\zeta_k}^*\circ P_{J_k}(y_k) - \rho| < 1/2^{k+1}$. For $k$ odd, by Remark \ref{some remark that lets you find a coordinate and interval providing the estimate that you more or less require}, we can choose $\zeta_k\in\Ga$ $\we(\zeta_k) = \we(\ga_k)$ (actually $\zeta_k = \ga_k$) and an interval $J_k$ of $\N$ so that $|e_{\zeta_k}^*\circ P_{J_k}(y_k) - \rho| < 7168m_{j_k}/n_{j_k}$ (where
$\we(\ga_k) = m_{j_k}^{-1}$), which, using \eqref{coding growth}, can be chosen to be below $1/2^{k+1}$. It is also important to note that for $k$ odd, $Ty_k = 0$. We conclude that the sequence $(\zeta_k, J_k)$ is comparable, for $k$ odd $e_{\zeta_k}^*\circ P_{J_k}(Ty_k) > 1/2$ and for $k$ even $e_{\zeta_k}^*\circ P_{J_k}(Ty_k) = 0$, hence, using an argument very similar to that in the proof of Proposition \ref{differences on dependent sequences blow up} we can find $n\inn$ so that the norm of $\sum_{k=1}^nTy_k$ is arbitrarily large, whereas by Proposition \ref{estimate on dependent sequence} (ii) $\|\sum_{k=1}^ny_k\|\leqslant 35840$. This means that $T$ is unbounded which completes the proof.
\end{proof}

\begin{rmk}\label{coordinates versus projections onto singletons}
We point out a fact that we will use to prove the next result. Recall that for each $n\inn$, $(d_\ga)_{\ga\in\De_n}$ is 2-equivalent to the unit vector basis of $\ell_\infty^n$. This easily implies the following: if $(x_k)_k$ is a block sequence in $\X$, then $\lim_k\sup_{\ga\in\Ga}|d_\ga^*(x_k)| = 0$ if and only if $\lim_k\sup_{n\inn}\|P_{\{n\}}x_k\| = 0$, where $P_{\{n\}}$ denotes the Bourgain-Delbaen projection onto the $n$-th coordinate of the FDD of $\X$.
\end{rmk}

\begin{prp}\label{scalar plus strictly singular}
For every bounded linear operator $T:\X\rightarrow\X$ there is a scalar $\la$, so that $T - \la I$ is strictly singular.
\end{prp}

\begin{proof}
We choose an accumulation point $\la$ of the sequence $(d_\ga^*(Td_\ga))_\ga$ and we will show that the operator $S = T - \la I$ is strictly singular. Passing to a subsequence $(d_{\ga_k})_k$ of the basis and adding a compact perturbation, we may assume that $d_{\ga_k}^*(Sd_{\ga_k}) = 0$ for all $k\inn$. By Proposition \ref{ss char}, it suffices to show that  $(Sd_{\ga_k})_k$ converges to zero in norm. Towards a contradiction, we assume that this is not the case. We shall follow steps similar to those used in the proof of Proposition \ref{ss char} to blow up the norm of $T$, in a slightly different way. More precisely, the goal is to use  Proposition \ref{dependent sequence on subsequence of the basis} to find a $(3584,\theta)$-dependent sequence $\{(\ga_k',y_k)\}_{k=1}^\infty$ and a sequence $\{(\zeta_k, J_k)\}_k$ with $\we(\zeta_k) = \we(\ga_k')$, $e_{\zeta_k}^*\circ P_{J_k}((-1)^kSy_k) > \e$, for some $\e>0$, and $e_{\zeta_k}^*\circ P_{J_k}(y_k) = 0$ for all $k\inn$. This last part in particular implies
that $\{(\zeta_k, J_k)\}_k$ is comparable and using the fact $e_{\zeta_k}^*\circ P_{J_k}((-1)^kSy_k) > \e$ and an argument similar to that used in the proof of Proposition \ref{differences on dependent sequences blow up} implies that one can find $n\inn$ so that the norm of $\sum_{k=1}^nSy_k$ is arbitrarily large, which in conjunction with Proposition \ref{estimate on dependent sequence} (ii) implies that $S$ is unbounded.

It remains to describe how to find $\{(\ga_k',y_k)\}_k$ and $\{(\zeta_k, J_k)\}_k$. By Proposition \ref{basis c0}, $(d_{\ga_k})_k$ admits only a $c_0$ spreading model and hence, the same is true for $(Sd_{\ga_k})_k$ which we may assume is a normalized block sequence. We distinguish two cases, namely whether $\lim_k\sup_{\ga\in\Ga}|d_\ga^*(Sd_{\ga_k})|$ is zero or not. We treat the second case, i.e. on some subsequence there are $\e >0$ and $(\eta_k)_k$ so that $d_{\eta_k}^*Sd_{\ga_k} > \e$ for all $k\inn$. As $d_{\ga_k}^*(Sd_{\ga_k}) = 0$, we obtain $\ga_k\neq\eta_k$ and keeping this in mind we can apply Proposition \ref{dependent sequence on subsequence of the basis} to find a $(3584,\theta)$-dependent sequence $\{(\ga_k',y_k)\}_{k=1}^\infty$ as in \eqref{capital control for the win} and a sequence $\{(\zeta_k, J_k)\}_k$, where $\zeta_k$ is built on $(d_{\eta_k}^*)_k$ using appropriate signs, $J_k = \ran Sd_{\ga_k}$,  with $\we(\zeta_k) = \we(\ga_k')$ and $e_{\zeta_k}^*\circ P_{J_k}((-1)^kSy_k) > \e\theta$.

Otherwise, $\lim_k\sup_{\ga\in\Ga}|d_\ga^*(d_{\ga_k})| = 0$. For each $k\inn$, we define $D_k^- = \ran Sd_{\ga_k}\cap[1,\ra(\ga_k))$ and $D_k^+ = \ran Sd_{\ga_k}\cap(\ra(\ga_k),\infty)$. Remark \ref{coordinates versus projections onto singletons} yields that either $\lim\sup_k\|P_{D_k^-}Sd_{\ga_k}\| > 0$ or $\lim\sup_k\|P_{D_k^+}Sd_{\ga_k}\| > 0$. We shall assume the first, set $D_k = D_k^-$ for all $k\inn$ and note that $\ra (d_{\ga_k})\notin D_k$ for all $k\inn$. Clearly, $\lim_k\sup_{\ga\in\Ga}|d_\ga^*(P_{D_k}Sd_{\ga_k})| = 0$, and since $(Sd_{\ga_k})_k$ only admits $c_0$ as a spreading model, we can deduce that $\al((P_{D_k}Sd_{\ga_k})_k) = 0$. We pass to a subsequence, and perhaps consider $-S$, to choose a sequence $(\eta_k)_k$ with $e_{\eta_k}^*(P_{D_k}Sd_{\ga_k}) > 3/4\|P_{D_k}Sd_{\ga_k}\|$ for all $k\inn$. Since $e_{\eta_k}^*\circ P_{D_k}d_{\ga_k} = 0$ for all $k$, we can use Remark \ref{first step building better} and Proposition \ref{dependent sequence on subsequence of the basis} to proceed as in
the previous case, that is, we can find a $(3584,\theta)$-dependent sequence $\{(\ga_k',y_k)\}_{k=1}^\infty$ as in \eqref{capital control for the win} and a sequence $\{(\zeta_k, J_k)\}_k$, where $\zeta_k$ is built on $(e_{\eta_k}^*\circ P_{D_k})_k$ using appropriate signs, with $\we(\zeta_k) = \we(\ga_k')$ and $e_{\zeta_k}^*\circ P_{J_k}((-1)^kSy_k) > \e$.
\end{proof}

\begin{rmk}
The same result can be proved, using similar arguments (see also \cite[Lemma 8.8]{AM2}), for operators $T:X\rightarrow X$, where $X$ is a block subspace of $\X$ generated by a block sequence $(x_k)_k$ which is either a subsequence of the basis, or satisfies $\adxk = 0$ as well as $\lim_k\sup_{\ga\in\Ga}|d_\ga^*(x_k)| = 0$.
\end{rmk}

\begin{rmk}\label{inclusion plus ss on subsequence of the basis}
The proof of Proposition \ref{scalar plus strictly singular} yields that if $(d_{\ga_k})_k$ is a subsequence of the basis of $\X$ and $Y = \overline{\langle\{d_{\ga_k}:k\inn\}\rangle}$, then every bounded linear operator $T:Y\rightarrow\X$ is a multiple of the inclusion plus a strictly singular operator.
\end{rmk}

Recall that a Banach space $X$ is called an $\ell_1$-predual if $X^*$ is isomorphic to $\ell_1$, or equivalently (\cite[Corollary, Page 182]{LS}), if $X$ is a $\mathscr{L}_\infty$-space with separable dual.

\begin{lem}\label{first step for compact operators}
Let $X$ be an $\ell_1$-predual, $T: X\rightarrow\X$ be a bounded linear operator, and assume that for every very fast growing sequence of \ac-averages $(b_k^*)_k$, $(T^*b_k^*)_k$ is norm-null. Then,
\begin{itemize}

\item[(i)] for every subsequence $(\ga_k)_k$ of $\Ga$, $(T^*d_{\ga_k}^*)_k$ is norm-null,

\item[(ii)] for every subsequence $(\ga_k)_k$ of $\Ga$ with $\{(\we(\ga_k))^{-1}:k\inn\}$ unbounded and successive intervals $(E_k)_k$ of $\N$, $(T^*(e_{\ga_k}^*\circ P_{E_k}))_k$ is norm null,

\item[(iii)] for every very fast growing sequence of \ac-averages $(b_k^*)_k$ and successive intervals $(E_k)_k$ of $\N$, $(T^*(b_{k}^*\circ P_{E_k}))_k$ is norm null.

\end{itemize}
\end{lem}

\begin{proof}
Note that the third statement immediately follows from Remark \ref{subseq preserve type}, which yields that $(b_k^*\circ P_{E_k})_k$ is a very fast growing sequence of \ac-averages as well. To see the proof of first two statements, note that if for each $k$, $x_k^* = d_{\ga_k}^*$ or $x_k^* = e_{\ga_k}^*\circ P_{E_k} = P^*_{E_k}e_{\ga_k}^*$, then $(x_k^*)_k$ is $w^*$-null and hence, so is $(T^*x_k^*)_k$. We conclude that if it is not norm-null, then it has a subsequence equivalent to the unit vector basis of $\ell_1$. In each case, either because elements of the basis always define \ac-averages, or using Proposition \ref{building averages}, one can find a very fast growing sequence of \ac-averages whose image under $T^*$ is not norm-null, contradicting the initial assumption.
\end{proof}

\begin{lem}\label{second step for compact operators}
Let $X$ be a Banach space and $T:X\rightarrow\X$ be a bounded and non-compact linear operator. Then there are a subsequence $(\ga_k)_k$ of $\Ga$ and a sequence of successive intervals $(E_k)_k$ of $\N$ so that $\lim\sup_k \|T^*(e_{\ga_k}^*\circ P_{E_k})\| > 0$.
\end{lem}

\begin{proof}
As $T$ is not compact, there is a normalized sequence $(x_k)_k$ and so that $(Tx_k)_k$ has no norm-convergent subsequence. Hence, passing if necessary to a subsequence, there are $\e > 0$ and a sequence of successive intervals $(E_k)_k$ of $\N$ so that $\|P_{E_k}Tx_k\| > \e$ for all $k$. Choose a sequence $(\ga_k)_k$ of $\Ga$ so that $e_{\ga_k}^*(P_{E_k}Tx_k) > \e$ for all $k$ and observe that it is the desired one.
\end{proof}

\begin{prp}\label{dual of non-compact preserves averages}
Let $X$ be an $\ell_1$-predual and $T:X\rightarrow\X$ be a bounded and non-compact linear operator. Then there exists a very fast growing sequence of \ac-averages $(b_k^*)_k$ so that $\lim\sup_k\|T^*b_k^*\| >0$.
\end{prp}

\begin{proof}
Let us assume that the conclusion is false, i.e. $T$ satisfies the assumptions of Lemma \ref{first step for compact operators}. We shall prove a statement, which in conjunction with Lemma \ref{second step for compact operators} and a finite inductive argument yields that $\|T\|$ is arbitrarily large, a contradiction. The statement is the following: if $(\ga_k)_k$ is a subsequence of $\Ga$, $(E_k)_k$ is a sequence of successive intervals of $\N$ and $\e>0$ so that $\|T^*(e_{\ga_k}^*\circ P_{E_k})\| > \e$ for all $k$, then there are a subsequence $(\eta_k)_k$ of $\Ga$ and a sequence of successive intervals $(F_k)_k$ of $\N$ so that $\|T^*(e_{\eta_k}^*\circ P_{F_k})\| > (m_1/2)\e$.

Let $(\ga_k)$, $(E_k)_k$ and $\e$ be as above. By Lemma \ref{first step for compact operators} (ii) and passing to a subsequence, there is $j_0$ so that $\we(\ga_k) = m_{j_0}^{-1}$ for all $k\inn$. Considering the evaluation analysis of each $\ga_k$, and passing to a subsequence, there is $a\inn$ so that $e_{\ga_k}^*\circ P_{E_k} = \sum_{r=1}^ad_{\xi_{k,r}}^*\circ P_{E_k} + (1/m_{j_0})\sum_{r=1}^ab_{k,r}^*\circ P_{E_k}$. As for each $k$ and $r$, $d_{\xi_{k,r}}^*\circ P_{E_k}$ is either $d_{\xi_{k,r}}^*$ or zero, Lemma \ref{first step for compact operators} (i) yields that $(T^*(\sum_{r=1}^ad_{\xi_{k,r}}^*\circ P_{E_k}))_k$ is norm-null. Recall that for each $k$, the sequence $(b_{k,r}^*)_{r=1}^a$ is very fast growing, hence Lemma \ref{first step for compact operators} (iii) implies that, passing to a subsequence, there is $\ell\inn$ so that for each $k$ there is $r_k$ so that $s(b_{k,r_k}) = \ell$ and $\|T^*((1/m_{j_0})b_{k,r_k}^*\circ P_{E_k})\| > \e/2$. Moreover, Lemma \ref{first step for compact operators}
 (i) implies that $b_{k,r}^*$ cannot be a basic average, hence $b_{k,r}^* = (1/\ell)\sum_{i=1}^s\e_ie_{\eta_{k,i}}^*\circ P_{F_{k,i}}$ with $s\leqslant \ell$. We conclude that for each $k$ there is $i_k$ so that $\|T^*e_{\eta_{k,i}}^*\circ P_{F_{k,i}\cap E_k}\| > (m_{j_0}/2)\e$ which completes the proof.
\end{proof}


\begin{thm}\label{spc property}
A bounded linear operator $T:\X\rightarrow\X$ is strictly singular if and only if it is compact. Hence, $\X$ satisfies the scalar-plus-compact property.
\end{thm}

\begin{proof}
Towards a contradiction, assume that there is a strictly singular operator $T:\X\rightarrow\X$ which is not compact. By Proposition \ref{dual of non-compact preserves averages} there is a very fast growing sequence $(b_k^*)_k$ of \ac-averages so that $\lim\sup_k\|T^*b_k^*\| > 0$. As the sequence $(b_k^*)_k$ is $w^*$-null, so is $(T^*b_k^*)_k$ and hence by a sliding hump argument we can pass to a subsequence and find a normalized block sequence $(x_k)_k$ in $\X$ so that $\lim\sup_kT^*b_k^*(x_k) > 0$. An argument in which each vector $x_k$ is split according to the weights of its local support (\cite[Definition 5.7]{AH}), yields that on some subsequence there are $\e > 0$, $C\geqslant 1$ and a $C$-RIS $(y_k)_k$, so that $b_k^*(Ty_k) = T^*b_k^*(y_k)>\e$ for all $k\inn$. To be more precise, the sequence $(y_k)_k$ is chosen so that is either has bounded local weights or rapidly decreasing local weights (\cite[Definition 5.9]{AH}) which yields that it is a RIS (\cite[Proposition 5.10]{AH}). For a more detailed
exposition of the argument see \cite[Proposition 5.11]{AH}. As $(y_k)_k$ is weakly null, we may assume that $(Ty_k)_k$ is a block sequence, i.e. we have found a $C$-RIS $(y_k)_k$, so that $(Ty_k)_k$ is a block sequence with $\al$-index positive. Combining Proposition \ref{adx positive} and \eqref{currant} of Proposition \ref{all basics in one} we obtain that passing to a subsequence, there is $\de>0$, so that for all $j \leqslant i_1 <\cdots <i_{n_j}$,
\begin{equation}
\frac{\de}{\|T\|} \leqslant  \frac{1}{\|T\|}\left\|\frac{m_j}{n_j}\sum_{k=1}^{n_j}Ty_{i_k}\right\| \leqslant \left\|\frac{m_j}{n_j}\sum_{k=1}^{n_j}y_{i_k}\right\| \leqslant 10C.
\end{equation}
By the above and Proposition \ref{on RIS are exact vectors and pairs}, we can construct a sequence $(u_k)_k$, with each $u_k$ a $(112C,k)$-exact vector, so that both $(u_k)_k$ and $(Tu_k)_k$ are seminormalized. By Remark \ref{on exact vectors biorthogonals and index zero} we obtain $\lim_k\sup_{\ga\in\Ga}|d_\ga^*(u_k)| = 0$ as well as  $\aduk = 0$. This contradicts Proposition \ref{ss char} (iii), since $T$ was assumed to be strictly singular.
\end{proof}

\begin{rmk}\label{compact char}
The above proof actually yields that if an operator $T:\X\rightarrow\X$ is non-compact then it maps some RIS to a sequence with $\al$-index positive (to be more precise, to a weakly null sequence that is a perturbation of a block sequences with $\al$-index positive).
\end{rmk}

\begin{rmk}
\label{some subspaces fail spc}
The space $\X$ does not have the scalar-plus-compact property hereditarily, i.e. there exists a subspace $Y$ of $\X$ and a strictly singular operator $T:Y\to Y$ that is not compact. This is also true for the space $\mathfrak{X}_{\mathrm{AH}}$ constructed in \cite{AH}. We repeat the argument for completeness. As it was explained in Remark \ref{comparing spreading models with dual methods}, the sequence $(y_q)_q$ with $y_q = \sum_{\ga\in\De_q}d_\ga$ generates an $\ell_1$ spreading model. If we set $Y$ to be the closed linear span of $(y_q)_q$, then by Proposition \ref{c0 sm in every subspace} there is a sequence in $Y$ generating a $c_0$ spreading model. By a theorem in \cite{AOST} there exists a strictly singular operator $S:Y\to Y$ that is not compact.
\end{rmk}

\section{Quotients of $\BmT$ with the scalar-plus-compact property}\label{Quotients of BmT with the scalar-plus-compact property}
Recall that the space $\X$ is defined using the tree $\mathcal{U}$ of all finite special sequences $\{(\ga_k,x_k)\}_{k=1}^d$ (see Subsection \ref{subsection tree}). A defining property of the tree $\mathcal{U}$ is that any of its maximal chains is infinite. This is precisely the reason why there are no boundedly complete sequences in the space $\X$.

In a way analogous to that in \cite{AM2}, for each ordinal numbers $2\leqslant \xi <\omega_1$, we can consider the tree $\mathcal{U}_\xi$ of all finite special sequences $\{(\ga_k,x_k)\}_{k=1}^d$ so that $\{\ra(\ga_k):k=1,\ldots,d\}\in\mathcal{S}_\xi$. Each such tree $\mathcal{U}_\xi$ defines a different class of \ac-averages which induce a self-determined subset $\Ga_\xi$ of $\bar{\Ga}$, resulting in a hereditarily indecomposable $\mathscr{L}_\infty$-space $\mathfrak{X}_\xi$ with the scalar plus compact property, which is a quotient of $\BmT$. We note that $\xi\geqslant 2$ is necessary to be able to prove the aforementioned properties.

As the tree $\mathcal{U}_\xi$ is well founded, it can be shown that the space $\mathfrak{X}_\xi$ is reflexively saturated, in particular its FDD is shrinking and every skipped block sequence in the space $\mathfrak{X}_\xi$ is boundedly complete. Therefore, if for some $\la>0$ and block subspace $X$ of $\mathfrak{X}_\xi$ we define the tree $\mathcal{N}$-$\mathcal{BC}_{\mathrm{sk}}(X,\la)$ of all skipped block sequences $(x_k)_{k=1}^d$ in $X$ satisfying $1\leqslant \|\sum_{k=n}^mx_k\| \leqslant \la$ for $1\leqslant n \leqslant m \leqslant d$, then $\mathcal{N}$-$\mathcal{BC}_{\mathrm{sk}}(X,\la)$ is well founded. We conclude that there is an ordinal number $\zeta_\xi$, so that the order of the tree $\mathcal{N}$-$\mathcal{BC}_{\mathrm{sk}}(X,\la)$ is at most $\zeta_\xi$ for every block subspace $X$ of $\mathfrak{X}_\xi$ and $\la > 0$. On the other hand, it can be deduced that for every block subspace $X$ of $\mathfrak{X}_\xi$ and $\eta < \xi$, the order of $\mathcal{N}$-$\mathcal{BC}_{\mathrm{sk}}(X,\la)$ is
at least $\eta$, for $\la$ sufficiently large. This easily yields that for every $\zeta > \zeta_\xi + 1$, the spaces $\mathfrak{X}_\xi$ and $\mathfrak{X}_\zeta$ are totaly incomparable. By passing to a co-final subset of the countable ordinal numbers and relabeling, we can find an uncountable family of pairwise totally incomparable Banach spaces $\{\mathscr{Y}_\xi:\xi<\omega_1\}$ so that each space $\mathscr{Y}_\xi$ is a hereditarily indecomposable and reflexively saturated $\mathscr{L}_\infty$-space with the ``scalar-plus-compact'' property which is moreover a quotient of $\BmT$.

A noteworthy fact is that the spaces in the family $\{\mathscr{Y}_\xi:\xi<\omega_1\}$ are defined using common weights and the same coding function, the difference between any two of them being that for each space a tree of special sequences with different complexity is used. It is also true that for $\zeta\neq\xi$ every bounded operator $T:\mathscr{Y}_\zeta\rightarrow\mathscr{Y}_\xi$ is compact. To see this, note that Proposition \ref{dual of non-compact preserves averages} also holds if $T:\mathscr{Y}_\zeta\rightarrow\mathscr{Y}_\xi$ is non-compact, hence arguing as in the proof Theorem \ref{spc property} one can find a sequence $(x_k)_k$ so that both it and $(Tx_k)_k$ are seminormalized and satisfy the assumptions of Proposition \ref{adx zero}. Using that $T$ is necessarily strictly singular, a construction can be carried out, similar to the one from the proof of Proposition \ref{ss char}, however some extra cases may need to be treated.

Summarizing the above we obtain the following result.

\begin{thm}\label{quotients with scalar plus compact}
There exists an uncountable family of pairwise totally incomparable Banach spaces $\{\mathscr{Y}_\xi:\xi<\omega_1\}$ satisfying the following:
\begin{itemize}

\item[(i)] Each space $\mathscr{Y}_\xi$ is a hereditarily indecomposable and reflexively saturated $\mathscr{L}_\infty$-space with the ``scalar-plus-compact'' property.

\item[(ii)] Each space $\mathscr{Y}_\xi$ is a quotient of $\BmT$.

\item[(iii)] For each $\xi\neq\zeta$, every bounded linear operator $T:\mathscr{Y}_\xi\rightarrow\mathscr{Y}_\zeta$ is compact.

\end{itemize}
\end{thm}

We recall that in \cite[Section 10.2]{AH} a continuum of pairwise incomparable spaces $\{X_a: a\in\mathfrak{c}\}$ is defined so that for $a\neq b$, every bounded operator $T:X_a\rightarrow X_b$ is compact. This is achieved by defining versions of the Argyros-Haydon space using almost disjoint families of weights and hence also different coding functions. All these spaces are actually quotients of $\BmT$ as well, hence the class $\{X_a: a\in\mathfrak{c}\}$ satisfies the conclusion of Theorem \ref{quotients with scalar plus compact}.

\begin{rmk}
As it was mentioned in Remark \ref{some further remarks on quotients}, a version $\tilde{\mathfrak{X}}_{\mathfrak{nr}}$ of $\X$ can be obtained as a quotient of a version $\tilde{\mathfrak{X}}_{\mathrm{AH}}$ of $\mathfrak{X}_{\mathrm{AH}}$. Actually, all spaces in the classes $\{\mathscr{Y}_\xi:\xi<\omega_1\}$ and $\{X_a: a\in\mathfrak{c}\}$ can be constructed to be quotients of one same Argyros-Haydon space $\mathfrak{X}_{\mathrm{AH}}$.
\end{rmk}

\section{Subspaces and quotients determined by self-determined subsets in $\X$}
In this section we very briefly describe some results concerning mainly quotients of $\X$. It is of some interest that one may find Banach spaces $X_1, X_2, X_3$, each one being a quotient of the previous one, so that $X_1$ and $X_3$ are reflexive saturated whereas $X_2$ contains no reflexive subspaces.

The following is the analogue of Proposition \ref{scalar plus compact on complements proposition} in the case of the space $\X$. The same argument used in that proof is necessary here as well, however some extra care needs to be taken.

\begin{prp}\label{inclusion plus compact on subspaces of script el infinity with no reflexive subspaces and the scalar plus compact property}
Let $\Ga'$ be a self-determined subset of $\Ga$ and let also $Y = \overline{\langle\{d_\ga:\ga\in\Ga\setminus\Ga'\}\rangle}$. Then every bounded linear operator $T:Y\rightarrow\X$ is a scalar multiple of the inclusion plus a compact operator.
\end{prp}

\begin{proof}
By Remark \ref{inclusion plus ss on subsequence of the basis}, it suffices to show that if $T$ is strictly singular then it is also compact. Recall that by Proposition \ref{complement gives subspace}, $Y$ is a $\mathscr{L}_\infty$-space, hence if $T$ is not compact, then the conclusion of Proposition \ref{dual of non-compact preserves averages} holds. As the set $\Ga'$ is self-determined, Proposition \ref{self-determinacy is all those nice things} (d) yields that the argument used in the proof of Theorem \ref{spc property} can be repeated, since when splitting a vector according to its local support, the components remain in the subspace $Y$.
\end{proof}

By combining Proposition \ref{inclusion plus compact on subspaces of script el infinity with no reflexive subspaces and the scalar plus compact property} and Lemma \ref{lemma about inclusion plus compact subsets} we observe that Corollary \ref{from bigger to smaller compact} also holds for the space $\X$. One can then perform a construction similar to that presented in Subsection \ref{sub self determined of AH}, using e.g. sizes instead of weights, to obtain a family $\{\Ga_\alpha: \alpha<\mathfrak{c}\}$ of subsets of $\Ga$ so that for $\alpha\neq\beta$ the set $\Ga_\alpha\setminus\Ga_\beta$ is infinite. This is achieved by choosing a family $\{L_\alpha: \alpha<\mathfrak{c}\}$ of infinite subsets of $\N$ with pairwise finite intersections. Each set $\Ga_\alpha$ is chosen so that its elements are built only using \ac-averages of sizes from the set $L_\alpha$. Corollary \ref{from bigger to smaller compact} (which as we explained holds for $\X$) yields the following.

\begin{prp}
There exists a continuum of pairwise non-isomorphic $\mathscr{L}_\infty$-subspaces $\{Y_a:\;a\in\mathfrak{c}\}$ of $\X$, each one of which has the scalar-plus-compact property. Moreover, for every $\alpha \neq \beta$ every bounded linear operator $T:Y_\alpha\rightarrow Y_\beta$ is compact.
\end{prp}

The situation becomes more involved when considering quotients of the space $\X$. The spaces $\mathscr{Y}_\xi$ defined in the previous section are reflexively saturated versions of the space $\X$, however they are not quotients of $\X$, in fact every bounded linear operator $T:\X\rightarrow\mathscr{Y}_\xi$ is compact. It is possible to obtain spaces similar to the spaces $\mathscr{Y}_\xi$, $\xi<\omega_1$ which are indeed quotients of $\X$, however the construction is a little more delicate. For $2\leqslant \xi <\omega_1$, consider the tree $\mathcal{U}_\xi$ as in Section \ref{Quotients of BmT with the scalar-plus-compact property}. Now consider the set of all those weights $(m_j,n_j)_{j\in L_\xi}$, which are weights of elements $\ga\in\bar{\Ga}$, that appear in elements (finite sequences of pairs) of $\mathcal{U}_\xi$. Now, construct a self-determined subset of $\Ga$, which  allows  to built averages and coordinates using precisely those weights. Note that within $\Ga$, the notion of a \ac-average is predetermined. The resulting quotient $\mathcal{Y}_\xi$ is a reflexively saturated $\mathscr{L}_\infty$ space with the scalar-plus-compact property. We observe that the conclusion of Theorem \ref{quotients with scalar plus compact} is false in the class of spaces $\{\mathcal{Y}_\xi:2\leqslant \xi\leqslant \omega_1\}$. The reason for this is that if $\xi < \zeta$ are such that $\mathcal{S}_\xi\subset\mathcal{S}_\zeta$, then $\mathcal{Y}_\xi$ is a quotient of $\mathcal{Y}_\zeta$. What is interesting however, is that although $\X$ contains no reflexive subspaces, it admits reflexively saturated quotients. Recall also that $\X$ is a quotient of a reflexively saturated space, e.g. $\BmT$ or of a version of $\mathfrak{X}_{\mathrm{AH}}$. Summarizing the preceding discussion we reach the conclusion stated in the result below, which ought to be compared to a classical theorem proved by Johnson and Zippin stating that every quotient of $c_0$ is isomorphic to a subspace of $c_0$ \cite{JZ}. Although quotients of classical $\mathscr{L}_\infty$-spaces have structure similar to those spaces, this is does not happen in non-classical $\mathscr{L}_\infty$-spaces.

\begin{thm}
There exists a triple of infinite dimensional Banach space $X_1$, $X_2$ and $X_3$ so that $X_1$ and $X_3$ are reflexively saturated, $X_2$ contains no reflexive subspace, $X_2$ is a quotient of $X_1$ and $X_3$ is a quotient of $X_2$. All three spaces are $\ell_1$-preduals with the scalar-plus-compact property.
\end{thm}

A classical result asserts that every quotient of a $C(K)$ space either is reflexive or it contains isomorphically $c_0$.  This invites the following question that, as far as we know, is open.
\begin{prb*}
Let $X$ be a $\mathscr{L}_\infty$-space  and $Y$ be a quotient of $X$. Does $Y$ have to  be reflexive or contain a $\mathscr{L}_\infty$-subspace? 
\end{prb*}

\end{document}